\documentclass[10pt,a4paper]{amsart}
\usepackage[english]{babel}
\usepackage{mathtools}
\usepackage{amsmath}
\usepackage{amssymb}
\usepackage{amsfonts}
\usepackage{amsthm}
\usepackage{amsxtra}
\usepackage{rotating}
\usepackage{nicefrac}
\usepackage{float}
\usepackage[dvips]{color}
\usepackage[usenames,dvipsnames]{pstricks}
\usepackage{pst-node}
\usepackage{pst-slpe}
\usepackage[all,web,arc,poly,dvips,color]{xy}
\usepackage{epic,eepic}
\usepackage{epsfig}
\usepackage{array}
\usepackage{bm}
\usepackage[sc,osf]{mathpazo}
\usepackage{pifont}
\usepackage{multirow}
\usepackage{latexsym}
\usepackage{graphics}
\usepackage{fullpage}
\usepackage{hyperref}
\usepackage{endnotes}
\usepackage{footmisc}
\usepackage[T1]{fontenc}
\usepackage[utf8]{inputenc}
\usepackage{xcolor}

\newcommand{\mathsym}[1]{{}}
\newcommand{\unicode}[1]{{}}

\theoremstyle{plain}
\newtheorem{theorem}{Theorem}[section]
\newtheorem{corollary}{Corollary}[section]
\newtheorem{lemma}{Lemma}[section]
\newtheorem{proposition}{Proposition}[section]
\theoremstyle{definition}
\newtheorem{definition}{Definition}[section]

\theoremstyle{remark}
\newtheorem{remark}{Remark}[section]

\renewcommand{\theequation}{\thesection.\arabic{equation}}

\newcommand{\C}{\mathbb C}
\newcommand{\R}{\mathbb R}
\newcommand{\Z}{\mathbb Z}
\newcommand{\N}{\mathbb N}

\newcommand{\half}{
        {\lower0.00ex\hbox{\raise.6ex\hbox{\the\scriptfont0 1}
                           \kern-.5em\slash\kern-.1em\lower.45ex
                                     \hbox{\the\scriptfont0 2}}}}
\newcommand{\quarter}{
        {\lower0.00ex\hbox{\raise.6ex\hbox{\the\scriptfont0 1}
                           \kern-.5em\slash\kern-.1em\lower.45ex
                                     \hbox{\the\scriptfont0 4}}}}
\newcommand{\tquarter}{
        {\lower0.00ex\hbox{\raise.6ex\hbox{\the\scriptfont0 3}
                           \kern-.5em\slash\kern-.1em\lower.45ex
                                     \hbox{\the\scriptfont0 4}}}}
\newcommand{\eighth}{
        {\lower0.00ex\hbox{\raise.6ex\hbox{\the\scriptfont0 1}
                           \kern-.5em\slash\kern-.1em\lower.45ex
                                     \hbox{\the\scriptfont0 8}}}}
\newcommand{\othird}{
        {\lower0.00ex\hbox{\raise.6ex\hbox{\the\scriptfont0 1}
                           \kern-.5em\slash\kern-.1em\lower.45ex
                                     \hbox{\the\scriptfont0 3}}}}

\begin{document}

\title[]{Singular Values of Products of Ginibre Random Matrices}

\author{N.S.~Witte}
\address{Institute of Fundamental Sciences, Massey University, Palmerston North 4442, New Zealand}
\email{\tt N.S.Witte@massey.ac.nz}
\author{P.J.~Forrester}
\address{Department of Mathematics and Statistics, University of Melbourne, Victoria 3010, Australia}
\email{\tt matpjf@ms.unimelb.edu.au}

\begin{abstract}
The squared singular values of the product of $M$ complex Ginibre matrices form a biorthogonal ensemble,
and thus their distribution is fully determined by a correlation kernel. The kernel permits a hard edge
scaling to a form specified in terms of certain Meijer G-functions, or equivalently hypergeometric functions
${}_0 F_M$, also referred to as hyper-Bessel functions. In the case $M=1$ it is well known that the corresponding
gap probability for no squared singular values in $(0,s)$ can be evaluated in terms of a solution of a particular
sigma form of the Painlev\'e III' system. One approach to this result is a formalism due to Tracy and Widom, involving
the reduction of a certain integrable system. Strahov has generalised this formalism to general $M \ge 1$, but has
not exhibited its reduction. After detailing the necessary working in the case $M=1$, we consider the problem of
reducing the 12 coupled differential equations in the case $M=2$ to a single differential equation for the
resolvent. An explicit 4-th order nonlinear is found for general hard edge parameters. For a particular choice
of parameters, evidence is given that this simplifies to a much simpler third order nonlinear equation. The small
and large $s$ asymptotics of the 4-th order equation are discussed, as is a possible relationship of the $M=2$ systems to
so-called 4-dimensional Painlev\'e-type equations.
\end{abstract}

\maketitle

\section[]{Introduction}\label{introduction}
\setcounter{equation}{0}

\subsection{Fredholm determinant}
Let $ X(1), \ldots, X(M) $, $ M \geq 1 $ be a sequence of rectangular matrices $ X(m) \in \C^{N_{m}\times N_{m-1}} $ 
with $ 1 \leq m \leq M $. We define the parameters $ \nu_m = N_m-N_0 $, $ m = 0,1, \ldots, M $ and will assume
that $ \nu_m \geq 0 $. Each of the $ X(m) $ are drawn from the Ginibre ensemble where their elements are i.i.d 
standard complex Gaussian random variables $ X(m)_{j,k} \in N[0,1]+i N[0,1] $ and each $ X(m) $ is independent 
of the others. We form the matrix product $ Y_M = X(M) \ldots X(1) \in \C^{N_M \times N_0} $ and the associated positive definite form
$ Y^{\dagger}_MY_M \in \C^{N_0 \times N_0} $. Our primary interest is in integrable structures, in
particular differential equations, characterising the smallest eigenvalue of $Y_M^\dagger Y_M$ in the so-called
hard edge limit. In the case $M = 1$ it is well known that the integrable structures relate to the Painlev\'e III equation
\cite{TW_1994a}, \cite{FW_2002a}.
Underlying the integrable structures is the explicit form of the joint distribution of all the eigenvalues, 
given by Akemann, Ipsen and Kieburg in 2013 for arbitrary $N_m$.
\begin{theorem}\cite{AIK_2013}\label{GinibreProdJPDF}
The squared singular values of $ Y_M $, $ Spec(Y^{\dagger}_MY_M) = (x_1,\ldots,x_{N_0}) $, 
form a determinantal point process on $ \R_{>0} $. This determinantal point
process is a bi-orthogonal ensemble with a joint probability density function (jPDF)
\begin{equation*}
   P^{(M)}(x_1, \ldots,x_{N_0}) = \frac{1}{Z_{N_0}} \prod_{1\leq j<k\leq N_0}(x_k-x_j)
                                  \det\left( w^{(M)}_{k-1}(x_j) \right)_{1\leq j,k \leq N_0} ,
\end{equation*}
where $ x_k\in \R_{>0} $, $ k=1, \ldots, N_0 $, $ Z_{N_0} $ is the normalisation constant, and the functions 
$ w^{(M)}_k $ are 
\begin{equation*}
   w^{(M)}_k(x) = G^{M,0}_{0,M}(x|\nu_M,\nu_{M-1},\ldots, \nu_2, \nu_1+k) ,
\end{equation*}
in terms of Meijer's $G$-function.
\end{theorem}
We denote the correlation kernel of the determinantal point process defined by the jPDF above by $ K_{N_0}^{(M)}(x,y) $,
meaning that the $n$-point correlation function
\begin{equation}
  \rho_{(n)}(x_1,\dots,x_n) = N_0(N_0-1)\cdots(N_0-n+1)
  \int_0^\infty P^{(M)}(x_1,\dots,x_n,x_{n+1},\dots,x_{N_0}) \, dx_{n+1} \cdots dx_{N_0}
\end{equation} 
is given by
\begin{equation}\label{K1}
  \rho_{(n)}(x_1,\dots,x_n) = \det [K_{N_0}^{(M)}(x_j,x_k)]_{j,k=1,\dots,n} .
\end{equation}
It turns out that for large $N_0$ the eigenvalues near the origin, referred to as the hard edge since the spectral density
is strictly zero for $x < 0$, are spaced on distances of order $1/N_0$. Scaling the eigenvalues by this factor and taking
$N_0 \to \infty$ whilst keeping the $\nu_m$ fixed defines the hard edge limit. The explicit form of the correlation kernel
in this limit was calculated by Kuijlaars and Zhang in 2014.
\begin{theorem}\cite{KZ_2014}
Let $ K_{N_0}^{(M)}(x,y) $ be the correlation kernel of the above determinantal point process.
Its hard edge scaled limit is given by
\begin{equation*}
   \lim_{N_0 \to \infty}\frac{1}{N_0}K_{N_0}^{(M)}\left( \frac{x}{N_0},\frac{y}{N_0} \right) = K_M(x,y) ,
\end{equation*} 
where 
\begin{equation}
  K_M(x,y) = \frac{\mathcal{B}\left( G^{1,0}_{0,M+1}(x|-\nu_0,-\nu_1,\ldots,-\nu_M), G^{M,0}_{0,M+1}(y|\nu_1,\ldots,\nu_M,\nu_0) \right)}{x-y} .
\label{HEkernel}  
\end{equation} 
Here $ \mathcal{B}(\cdot,\cdot) $ is a bilinear operator defined by
\begin{equation*}
   \mathcal{B}(f(x),g(y)) \coloneqq (-1)^{M+1}\sum^{M}_{j=0}(-1)^j \left( x\frac{d}{dx} \right)^j f(x) \sum^{M-j}_{i=0} \alpha_{i+j} \left( y\frac{d}{dy} \right)^i g(y) .
\end{equation*} 
The constants $ \alpha_i $ are determined from
\begin{equation*}
   \prod^{M}_{m=0}(x-\nu_m) = x\sum^{M}_{i=0} \alpha_i x^i . 
\end{equation*} 
The kernel functions $ f,g $ are defined in terms of the Meijer $G$-functions by
\begin{equation*}
f(x) = G^{1,0}_{0,M+1}(x|-\nu_0,-\nu_1, \ldots, -\nu_M) ,
\quad
g(y) = G^{M,0}_{0,M+1}(y|\nu_1,\ldots, \nu_M,\nu_0) .
\end{equation*}
\end{theorem}

The Meijer G-functions are specified in terms of certain Mellin-Barnes integrals. More important to us is the fact that
they satisfy certain linear differential equations of degree $M+1$ \cite{BS_2013}.
\begin{proposition}\label{MeijerG_ode}
The functions $ f $ and $ g $ satisfy the linear differential equations
\begin{equation*}
   \prod^{M}_{j=0} \left( x\frac{d}{dx}+\nu_j \right)f(x) = -xf(x) ,
\qquad
   \prod^{M}_{j=0} \left( y\frac{d}{dy}-\nu_j \right)g(y) = (-1)^Myg(y) .
\end{equation*}
\end{proposition}
\noindent
Also useful in the ensuing theory are the sequence of related functions which we define for 
$ 0 \leq j \leq M $
\begin{equation}
   \phi_j(x) = (-1)^{M-j+1}\left( x\frac{d}{dx} \right)^j f(x) , \qquad
   \psi_j(y) = \sum^{M-j}_{i=0}\alpha_{i+j}\left( y\frac{d}{dy} \right)^i g(y) .
\label{KernelFns}
\end{equation}
Thus the above kernel \eqref{HEkernel} can be written as a generalised ``integrable`` kernel, see e.g. \cite{KZ_2014},
\begin{equation*}
  K_M(x,y) = \frac{\sum_{j=0}^M \phi_j(x)\psi_j(y)}{x-y}, \qquad \sum_{j=0}^M \phi_j(x)\psi_j(x) = 0 .
\end{equation*} 
We remark that the latter orthogonality relation is far from obvious, given the definitions made.

We now come to the central objects of our study, the hard edge gap probabilities. Let $0 \le a_1 < a_2 < \cdots < a_{2L-1} < a_{2L} < \infty$
be the endpoints of a collection of $L$ intervals of $\mathbb R_+$, and denote their union $J = \cup_{l=1}^L (a_{2l-1},a_{2l}) $. The probability
that there are no eigenvalues in $J$ is referred to as the gap probability and denoted $E_M(0;J)$. A standard result for a determinant point
process tells us that (see e.g.~\cite[\S 9.1]{For_2010})
\begin{equation}
   E_M(0;J) = \det(\mathbb 1 - \mathbb K_M) ,
\label{HE_FredholmDet}
\end{equation}
where $\mathbb K_M$ is an integral operator acting on $L^2((0,\infty))$ with kernel $K_M(x,y) \chi_J(y)$, where $K_M$
is given by \eqref{HEkernel} and $ \chi_J(y)$ is the characteristic function of the interval $J$.

\subsection{Strahov's extension of Tracy-Widom theory}\label{StrahovTheory}

In distinction to the case $M=1$, the kernel \eqref{HEkernel} for $M \ge 2$ is not symmetric.
Thus, in addition to the integral kernel
\begin{equation}
  {\mathbb K}_M \stackrel{.}{=} K_M(x,y) \chi_J(y) ,
\end{equation}
(here the symbol $ \stackrel{.}{=} $ denotes ''with kernel``), one requires the additional integral operators
\begin{equation*}
  \mathbb{K}^{\prime}_M \stackrel{.}{=} K_M(y,x)\chi_J(y) ,
\qquad
  \mathbb{K}^T_M \stackrel{.}{=} K_M(y,x)\chi_J(x) .
\end{equation*} 
From these integral operators we define the primary variables $ 0 \leq m \leq M $ and $ 1 \leq l \leq L $  
\begin{gather*}
  x^{(2l)}_m  \coloneqq \sqrt{-1}(1-\mathbb{K}_M)^{-1} \phi_m(a_{2l}) ,
\qquad
  y^{(2l)}_m  \coloneqq \sqrt{-1}(1-\mathbb{K}^{\prime}_M)^{-1} \psi_m(a_{2l}) ,
\\
  x^{(2l-1)}_m  \coloneqq (1-\mathbb{K}_M)^{-1} \phi_m(a_{2l-1}) ,
\qquad
  y^{(2l-1)}_m  \coloneqq (1-\mathbb{K}^{\prime}_M)^{-1} \psi_m(a_{2l-1}) .
\end{gather*}
This essentially means a doubling of the number of primary variables over that occurring in the Tracy and Widom theory.
Furthermore we require the auxiliary variables, which are constructed as inner products of the primary 
variables
\begin{align*}
  \xi_m & \coloneqq (-1)^M \sum^{L}_{l=1} \int^{a_{2l}}_{a_{2l-1}} dx\; \phi_0(x) \left( 1-\mathbb{K}^{\prime}_M \right)^{-1} \psi_m(x) + (-1)^{M+1-m}e_{M+1-m}(\nu_0, \ldots, \nu_M) ,
\\
  \eta_m & \coloneqq (-1)^M \sum^{L}_{l=1} \int^{a_{2l}}_{a_{2l-1}} dx\; \phi_m(x) \left( 1-\mathbb{K}^{\prime}_M \right)^{-1} \psi_M(x) .
\end{align*}
Here $e_k(\{x\})$ denotes the $k$-th elementary symmetric polynomial in the variables $\{x\}$.
It follows that the gap probability for $ M \geq 1 $, and say for a single interval $ J=(0,s) $, i.e. $ L=1 $, is 
determined by a certain product of the primary variables\footnote{This differs from the formulae of Prop. 3.9 and 
\S 4.5 of \cite{Str_2014} in the sign of the integral. This is due to the omission of $ \sqrt{-1} $ factors in the 
relations following the first paragraph at the beginning of \S 4.3, when substituted into Eq. (4.42) of that work.} 
\begin{equation}
   \det\left( 1-\mathbb{K}_M \right) = \exp\left\{ (-1)^{M+1} \int^{s}_{0}dt \log(\frac{s}{t})x_0(t)y_M(t) \right\} 
                                     = \exp\left\{ \int^{s}_{0}dt\, t^{-1}\eta_0(t) \right\} .
\label{tauM}
\end{equation} 

Strahov also observed that the foregoing system is a Hamiltonian system with $ 2L $ Hamiltonians $ H_j $ and 
$ (2L+1)(M+1) $ canonical conjugate pairs of co-ordinates $ x^{(k)}_m, y^{(k)}_m $ and $ \xi_m, \eta_m $. 
The Hamiltonian equations of motion are then 
\begin{equation*}
   a_{j}\frac{\partial}{\partial a_{j}} x^{(k)}_m = \frac{\partial H_j}{\partial y^{(k)}_m}, \quad
   a_{j}\frac{\partial}{\partial a_{j}} y^{(k)}_m = -\frac{\partial H_j}{\partial x^{(k)}_m}
\end{equation*} 
and
\begin{equation*}
   \frac{\partial}{\partial a_{j}} \xi_m = \frac{\partial H_j}{\partial \eta_m}, \quad
   \frac{\partial}{\partial a_{j}} \eta_m = -\frac{\partial H_j}{\partial \xi_m}   
\end{equation*} 
for $ 1 \leq j,k \leq 2L $ and $ 0 \leq m \leq M $. The Hamiltonians are given explicitly by
\begin{multline}
   H_j = -x_0^{(j)}\left( \sum^M_{m=0}\eta_my^{(j)}_m \right) + \left( \sum^{M}_{m=0}\xi_mx^{(j)}_m \right)y^{(j)}_M + (-1)^{M+1}a_jx^{(j)}_0y^{(j)}_M
\\
         -\sum^{M-1}_{m=0}x^{(j)}_{m+1}y^{(j)}_{m} + \sum^{2L}_{k=1, k \neq j}\frac{a_k}{a_j-a_k}\sum^M_{m',m=0} x^{(j)}_mx^{(k)}_{m'}y^{(k)}_my^{(j)}_{m'} .
\label{M_Ham}
\end{multline}

Also noted in \cite{Str_2014} is the fact that this Hamiltonian system is an isomonodromic system with
a natural representation as $ (M+1)\times(M+1) $ matrices. One makes the following definitions,
\begin{equation*}
  E = (-1)^{M+1}
    \begin{pmatrix}
       0 & 0 & \ldots & 0 \cr
       0 & 0 & \ldots & 0 \cr
       \vdots & & & 0 \cr
       1 & 0 & \ldots & 0 \cr
    \end{pmatrix} ,\quad
   C = 
   \begin{pmatrix}
      -\eta_0 & -1 &  0 & \ldots & 0 \cr
      -\eta_1 &  0 & -1 & \ldots & 0 \cr
       \vdots &    &    &        & \vdots \cr
      -\eta_{M-1}  &  0 &  0 & \ldots & -1 \cr
      -\eta_{M}+\xi_0   & \xi_1 & \xi_2 & \ldots & \xi_M \cr
    \end{pmatrix} ,
\end{equation*}
and constructs the residue matrices thus
\begin{equation*}
   A^{(l)} =
   \begin{pmatrix}
     x^{(l)}_0 \cr
     x^{(l)}_1 \cr
     \vdots \cr
     x^{(l)}_M \cr
   \end{pmatrix} \otimes
   \begin{pmatrix}
     y^{(l)}_0, & y^{(l)}_1, & \ldots, & y^{(l)}_M \cr
   \end{pmatrix} .
\end{equation*} 
Then the first member of the Lax pair for $ \Psi(z;a_1,\ldots,a_{2L}) $ is
\begin{equation}
  \frac{\partial \Psi}{\partial z} = \left\{ E + \frac{C-\sum^{2L}_{j=1}A^{(j)}}{z} + \sum^{2L}_{j=1}\frac{A^{(j)}}{z-a_j} \right\}\Psi ,
\label{1st_Lax}
\end{equation} 
and the second members are for $ 1\leq j \leq 2L $ 
\begin{equation}
  \frac{\partial \Psi}{\partial a_j} = -\frac{A^{(j)}}{z-a_j} \Psi .
\label{2nd_Lax}
\end{equation} 
The compatibility relations of \eqref{1st_Lax} and \eqref{2nd_Lax} now leads to Schlesinger equations, which
are precisely the same as those derived from the Hamilton equations of motion.

\subsection{Plan of the paper}
In Section \ref{M=1theory} we detail the analysis required to reduce the Hamiltonian system in the case $M=1$ down 
to a single nonlinear equation characterising the Hamiltonian and thus the gap probability in the case $L=1$.
This characterisation is a known result \cite{TW_1994a}, but its derivation via the formalism of \S \ref{StrahovTheory} 
involves some subtleties, the appreciation of which is essential to progress to the new territories of $M \ge 2$.
The case $M=2$ is addressed in Section \ref{M=2theory}. The corresponding Hamiltonian system consists of 12
coupled equations. Calling on the experience gained from Section \ref{M=1theory}, and with the essential aid of 
computer algebra, a reduction is found of the 12 coupled equations
down to a single nonlinear equation determining the gap probability. This equation is of fourth order, and is
given in Proposition \ref{eta_form}. Both the small and large $s$ asymptotics of this equation can be determined,
and from the latter the corresponding large spacing asymptotic form of the gap probability is deduced; see
Corollary \ref{asymptoticGap}. In the special case $\nu_1=-1/2$, $\nu_2=0$ evidence is found that the fourth order
equation of Proposition \ref{eta_form} can be reduced to a specific third order equation, \eqref{3rdOrderODE} below.
We conclude by discussing a possible relationship of the $M=2$ systems to the recently introduced theory of 
so-called 4-dimensional Painlev\'e equations.

\section{\texorpdfstring{$ M=1 $}{1ST} Tracy-Widom Theory at the Hard Edge}\label{M=1theory}
\setcounter{equation}{0}

The original Tracy and Widom theory must be equivalent to the $ M=1 $ and $ L=1 $ case of the preceding
theory, although this is not immediate. Therefore it is instructive to consider this case first,
primarily because it will provide essential guidance for the $ M \geq 2 $ cases.
This will also serve to clarify some misunderstanding present in the existing literature relating to this point.
  
From Prop. 3.9 of \cite{Str_2014} for $ J=(0,s) $, $ a_1=0, a_2=s $, i.e $ x_{j} = x^{(2)}_{j} $, $ y_{j} = y^{(2)}_{j} $
and $M=1$, we read off the following system of coupled quasi-linear ODEs ($' = d/ds$) with respect to $ s $
\begin{align}
  s x_{0}' & = -\eta_{0}x_{0}-x_{1} ,
\label{M=1ODE:1} \\
  s x_{1}' & = -\eta_{1}x_{0}+sx_{0}+\xi_{0}x_{0}+\xi_{1}x_{1} ,
\label{M=1ODE:2} \\
  s y_{1}' & = -\xi_{1}y_{1}+y_{0} ,
\label{M=1ODE:3} \\
  s y_{0}' & = -\xi_{0}y_{1}-sy_{1}+\eta_{0}y_{0}+\eta_{1}y_{1} ,
\label{M=1ODE:4} \\
  \xi_{0}' & = x_{0}y_{0} ,
\label{M=1ODE:5} \\
  \xi_{1}' & = x_{0}y_{1} ,
\label{M=1ODE:6} \\
  \eta_{0}' & = x_{0}y_{1} ,
\label{M=1ODE:7} \\
  \eta_{1}' & = x_{1}y_{1} .
\label{M=1ODE:8}
\end{align}
In this case the Hamiltonian \eqref{M_Ham} simplifies to
\begin{equation*}
  H = -\eta_{0}x_{0}y_{0}+(\xi_{0}-\eta_{1}+s)x_{0}y_{1}-x_{1}y_{0}+\xi_{1}x_{1}y_{1} ,
\end{equation*}
and the Hamiltonian equations of motion
\begin{equation*}
  s x_{j}' = \frac{\partial}{\partial y_{j}}H, \quad  s y_{j}' = -\frac{\partial}{\partial x_{j}}H, \quad j=0,1
\end{equation*}
\begin{equation*}
  \eta_{j}' = \frac{\partial}{\partial \xi_{j}}H, \quad  \xi_{j}' = -\frac{\partial}{\partial \eta_{j}}H, \quad j=0,1
\end{equation*}
furnish the system \eqref{M=1ODE:1}-\eqref{M=1ODE:8} above.
Note that \eqref{M=1ODE:7} substituted in \eqref{tauM} with $M=1$ shows
\begin{equation}
  \det(\mathbb 1 - \mathbb K_1) = \exp \Big( \int_0^s dt\, \frac{\eta_0(t)}{t} \Big) .
\label{M=1tau}
\end{equation}

In the matrix formulation of the isomonodromic problem we recall the definitions
\begin{equation*}
   E \coloneqq \begin{pmatrix} 0 & 0 \\ 1 & 0 \end{pmatrix}, \quad
   C \coloneqq \begin{pmatrix} -\eta_{0} & -1 \\ \xi_{0}-\eta_{1} & \xi_{1} \end{pmatrix} ,
\end{equation*}
and
\begin{equation*}
   A \coloneqq A^{(2)} = \begin{pmatrix} x_{0}y_{0} & x_{0}y_{1} \\ x_{1}y_{0} & x_{1}y_{1} \end{pmatrix}
     = \begin{pmatrix} x_{0} \\ x_{1} \end{pmatrix} \otimes \begin{pmatrix}y_{0} & y_{1} \end{pmatrix} , 
\end{equation*}
where $ A^{(2)} $ is a rank $ 1 $ matrix so $ \det A^{(2)} = 0 $. The Schlesinger equations are now
\begin{equation}
s A^{(2)\prime} = \left[ C+sE, A^{(2)} \right] ,
\qquad
C^{\prime} = \left[ E, A^{(2)} \right] .
\end{equation}

\begin{proposition}
The isomonodromic system $ \Psi(x,s) $ has a singularity pattern $ \frac{3}{2}\!+\!1\!+\!1 $ where the 
Riemann-Papperitz symbol is
\begin{equation}
   \left\{ \begin{array}{cccc}
            0 & 1 & \multicolumn{2}{c}{\infty(\frac{1}{2})} \\
            -\nu_0 & 0 & i\sqrt{s} & -\frac{1}{2} \\
            -\nu_1 & 0 & -i\sqrt{s} & \nu_0+\nu_1
           \end{array}
   \right\} .
\label{M=1RPsymbol}
\end{equation}
We have the resonant or ramified case, see \cite{KH_1999}, \cite{Kap_2002} and \cite{OO_2006}.
\end{proposition}
\begin{proof}
The isomonodromic system \eqref{M=1RPsymbol} differs from the one in \eqref{1st_Lax} and \eqref{2nd_Lax} through the 
transformation of the independent variable $ z \mapsto s z $ and $ \Psi(sz,s) \mapsto \Psi(z,s) $, which become
\begin{gather}
\frac{\partial \Psi}{\partial z} = \left\{ sE + \frac{C-A^{(1)}}{z} + \frac{A^{(1)}}{z-1} \right\}\Psi ,
\label{iso_M=1}\\
\frac{\partial \Psi}{\partial s} = \left\{ s^{-1}E z + s^{-1}C \right\} \Psi .
\end{gather}
The effect of this is to place the regular singularities at the canonical positions $ 0 $ and $ 1 $.
The resonant or ramified case arises because $ E $ is nilpotent with eigenvalues $ 0, 0 $; the eigenvalues of 
$ C-A^{(1)} $ are $ -\nu_0, -\nu_1 $ whilst those of $ A^{(1)} $ are $ 0, 0 $. Let us denote the matrix in braces
on the right-hand side of \eqref{iso_M=1} by $ A $. The Jordan decomposition of $ sE $ is
\begin{equation*}
     sE = \begin{pmatrix} 0 & s^{-1} \\ 1 & 0 \end{pmatrix}\cdot
          \begin{pmatrix} 0 & 1 \\ 0 & 0 \end{pmatrix}\cdot
          \begin{pmatrix} 0 & s^{-1} \\ 1 & 0 \end{pmatrix}^{-1} .
\end{equation*}
so we transform the system \eqref{iso_M=1} to 
$ B = \begin{pmatrix} 0 & s^{-1} \\ 1 & 0 \end{pmatrix}^{-1}\cdot A \cdot\begin{pmatrix} 0 & s^{-1} \\ 1 & 0 \end{pmatrix} $.
We next apply the shearing transformation $ S \coloneqq {\rm diag}(1,z^{-g}) $ with an arbitrary exponent $ g $
and form a new coefficient matrix $ C = S^{-1}\cdot B\cdot S - S^{-1}\cdot S' $. The leading exponent matrix of $ C $ 
is $ \begin{pmatrix} 1 & g \\ -g+1 & 1 \end{pmatrix} $ so we achieve off-diagonal balance if we choose $ g=1/2 $.
Under this choice the leading coefficient of $ C $ (of order $ z^{-1/2} $) is now diagonalisable if $ s \neq 0 $
\begin{equation*}
   \begin{pmatrix} 0 & 1 \\ -s & 0 \end{pmatrix} =
   \begin{pmatrix} -is^{-1/2} & is^{-1/2} \\ 1 & 1 \end{pmatrix}\cdot 
   \begin{pmatrix} is^{1/2} & 0 \\ 0 & -is^{1/2} \end{pmatrix}\cdot
   \begin{pmatrix} -is^{-1/2} & is^{-1/2} \\ 1 & 1 \end{pmatrix}^{-1} .
\end{equation*}

So we apply this transformation and define another coefficient matrix 
$ D = \begin{pmatrix} -is^{-1/2} & is^{-1/2} \\ 1 & 1 \end{pmatrix}^{-1}\cdot C\cdot\begin{pmatrix} -is^{-1/2} & is^{-1/2} \\ 1 & 1 \end{pmatrix} $. Associated with the fractional exponent for $ g $ we define a new spectral variable $ z=\frac{1}{4}w^2 $, and
perform a large $ w $ expansion
\begin{equation*}
   \frac{1}{2}w D = 
   \begin{pmatrix} is^{1/2} & 0 \\ 0 & -is^{1/2} \end{pmatrix}
   +w^{-1}\begin{pmatrix} 1/2-e_1 & 1/2+e_1-2\eta_0 \\ 1/2+e_1-2\eta_0 & 1/2-e_1 \end{pmatrix}
   +{\rm O}(w^{-2}) .
\end{equation*}
The sub-leading term appearing above can also be diagonalised
\begin{equation*}
   \begin{pmatrix} 1/2-e_1 & 1/2+e_1-2\eta_0 \\ 1/2+e_1-2\eta_0 & 1/2-e_1 \end{pmatrix} =
   \begin{pmatrix} 1 & -1 \\ 1 & 1 \end{pmatrix}\cdot
   \begin{pmatrix} 1-2\eta_0 & 0 \\ 0 & -2e_1+2\eta_0 \end{pmatrix}\cdot
   \begin{pmatrix} 1 & -1 \\ 1 & 1 \end{pmatrix}^{-1} ,
\end{equation*}
and these diagonal elements give us the last column of the Riemann-Papperitz symbol.
\end{proof}

The ramified cases of the isomonodromic systems are quite important because they arise very naturally from
random matrix theory applications, and the re-interpretation of the degeneration scheme of the Painlev\'e 
equations via isomonodromy deformations was completed relatively recently by Kapaev \& Hubert 1999 \cite{KH_1999}, 
Kapaev 2002 \cite{Kap_2002} and Ohyama and Okumura 2006 \cite{OO_2006}. In this expanded scheme, 
see Fig. \ref{2accessoryscheme}, there are 5 integer (unramified) types and 5 half-integer types, even though 
there are only 6 independent transcendents.

\begin{figure}[H]
\resizebox{1.0\textwidth}{!}{
\begin{minipage}{\textwidth}
\begin{xy}
{(3,0) *{\begin{tabular}{|c|}
\hline
1+1+1+1\\
\hline
$P_{\rm VI}$\\
\hline
\end{tabular}
}},
{(35,0) *{\begin{tabular}{|c|}
\hline
2+1+1\\
\hline
$P_{\rm V}$\\
\hline
\end{tabular}
}},
{\ar (14,0);(26,0)},
{\ar (44,0);(57,14)},
{\ar@[red] (44,0);(57,0)},
{\ar (44,0);(57,-14)},
{(67,15) *{\begin{tabular}{|c|}
\hline
2+2\\
\hline
$P_{\rm III}(D_6)$\\
\hline
\end{tabular}}},
{(67,0) *{\begin{tabular}{|c|}
\hline
\textcolor{blue}{$\frac{3}{2}+1+1$}\\
\hline
\textcolor{blue}{${\rm deg}-P_{\rm V}$}\\
\hline
\end{tabular}}},
{\ar@[red] (78,14);(92,14)},
{\ar (78,14);(92,1)},
{\ar (78,0);(92,13)},
{\ar (78,0);(92,-13)},
{\ar (78,-14);(92,-1)},
{\ar@[red] (78,-14);(92,-14)},
{(67,-15) *{\begin{tabular}{|c|}
\hline
3+1\\
\hline
$P_{\rm IV}$\\
\hline
\end{tabular}}},
{(102,15) *{\begin{tabular}{|c|}
\hline
\textcolor{blue}{$2+\frac{3}{2}$}\\
\hline
\textcolor{blue}{$P_{\rm III}(D_7)$}\\
\hline
\end{tabular}}},
{(102,0) *{\begin{tabular}{|c|}
\hline
4\\
\hline
$P_{\rm II}$\\
\hline
\end{tabular}}},
{\ar (112,14);(125,1)},
{\ar@[red] (112,14);(125,14)},
{\ar@[red] (112,0);(125,0)},
{\ar (112,-14);(125,-1)},
{(102,-15) *{\begin{tabular}{|c|}
\hline
\textcolor{blue}{$\frac{5}{2}+1$}\\
\hline
\textcolor{blue}{$P_{\rm 34}$}\\
\hline
\end{tabular}}},
{(135,15) *{\begin{tabular}{|c|}
\hline
\textcolor{blue}{$\frac{3}{2}+\frac{3}{2}$}\\
\hline
\textcolor{blue}{$P_{\rm III}(D_8)$}\\
\hline
\end{tabular}}},
{(135,0) *{\begin{tabular}{|c|}
\hline
\textcolor{blue}{$\frac{7}{2}$}\\
\hline
\textcolor{blue}{$P_{\rm I}$}\\
\hline
\end{tabular}}},
\end{xy}
\end{minipage}}
\caption{The degeneration scheme of the Painlev\'e equations interpreted through their isomonodromic
	deformation problems. The unramified and ramified cases are given in black and blue entries respectively,
	and the singularity confluence transitions are given by black arrows, while the drop in the Poincar\'e
	index transitions (in this case always $ 1/2$) are given by red arrows. The ${\rm deg}-P_{\rm V}$ system is 
	equivalent to the $P_{\rm III}(D_6)$ system, whilst $P_{\rm 34}$ is equivalent to $P_{\rm II}$.}
\label{2accessoryscheme}
\end{figure}
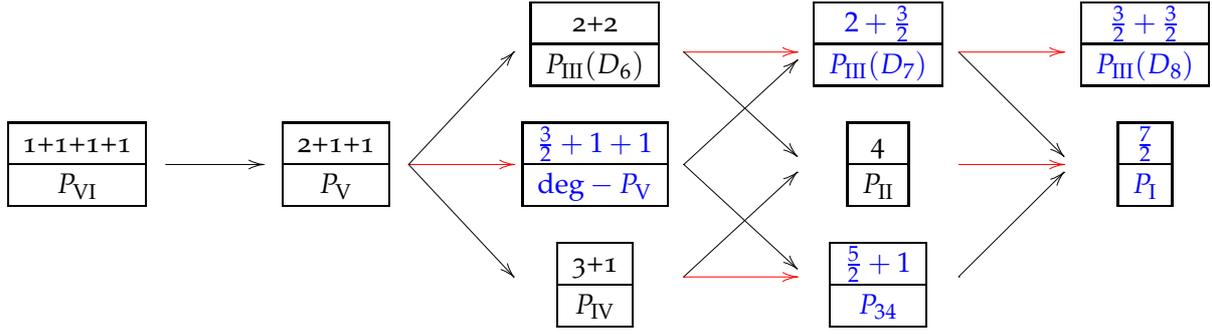

In the case $M=1$ we read off from \eqref{HEkernel}, together with knowledge of special cases of the Meijer G-function
(see e.g.~\cite{Luke_1969a}), that
\begin{gather*}
   f(x) = G^{1,0}_{0,2}(x|-\nu_{0},-\nu_{1}) = x^{-\frac{1}{2}(\nu_0+\nu_1)}J_{\nu_1-\nu_0}(2\sqrt{x}) ,
\\
   g(x) = G^{1,0}_{0,2}(x|\nu_{1},\nu_{0}) = x^{\frac{1}{2}(\nu_0+\nu_1)}J_{\nu_1-\nu_0}(2\sqrt{x}) .
\end{gather*}
Recalling \eqref{KernelFns}, and making use of difference-differential identities for the Bessel functions we 
thus have
\begin{align*}
   \phi_0(x) & = x^{-\frac{1}{2}(\nu_0+\nu_1)}J_{\nu_1-\nu_0}(2\sqrt{x}) ,
\\
   \phi_1(x) & = \nu_0x^{-\frac{1}{2}(\nu_0+\nu_1)}J_{\nu_1-\nu_0}(2\sqrt{x})+x^{-\frac{1}{2}(\nu_0+\nu_1-1)}J_{\nu_1-\nu_0+1}(2\sqrt{x}) ,
\\
   \psi_0(x) & = -\nu_0x^{\frac{1}{2}(\nu_0+\nu_1)}J_{\nu_1-\nu_0}(2\sqrt{x})-x^{\frac{1}{2}(\nu_0+\nu_1+1)}J_{\nu_1-\nu_0+1}(2\sqrt{x}) ,
\\
   \psi_1(x) & = x^{\frac{1}{2}(\nu_0+\nu_1)}J_{\nu_1-\nu_0}(2\sqrt{x}) .
\end{align*}

\begin{remark}\label{M=1split}
It is obvious from the above that the linear orthogonality relation
\begin{equation*}
 \psi_0(x) \phi_0(x)+\psi_1(x) \phi_1(x) = 0 ,
\end{equation*}
holds and in fact one observes the splitting or folding relations in this case
$ \psi_1(x) = x^{\nu_0+\nu_1}\phi_0(x) $ and $ \phi_1(x) = -x^{-\nu_0-\nu_1}\psi_0(x)  $.
\end{remark}

The inner product functions in the case $M=1$ are
\begin{gather*}
    V_{0,0}(s) = \int^{s}_0 dx \phi_0(x) \int^{\infty}_0 dz (1-\mathbb{K}_1\chi)^{-1}(x,z)\psi_0(z) ,
\\
    V_{0,1}(s) = \int^{s}_0 dx \phi_0(x) \int^{\infty}_0 dz (1-\mathbb{K}_1\chi)^{-1}(x,z)\psi_1(z) ,
\\
    V_{1,1}(s) = \int^{s}_0 dx \phi_1(x) \int^{\infty}_0 dz (1-\mathbb{K}_1\chi)^{-1}(x,z)\psi_1(z) ,
\end{gather*}
and these have the expansions around $ s=0 $
\begin{gather*}
   V_{0,0}(s) \mathop{\sim}\limits_{s \to 0} -V_{1,1}(s) \mathop{\sim}\limits_{s \to 0} -\frac{s^{\nu+2}}{\Gamma(\nu+1)\Gamma(\nu+3)} ,
\\   
   V_{0,1}(s) \mathop{\sim}\limits_{s \to 0} \frac{s^{\nu+1}}{\Gamma(\nu+1)\Gamma(\nu+2)} .
\end{gather*}
Using the Bessel function integral identities
\begin{gather*}
   \int^x_0 du u^2J_{\nu}(u)J_{\nu+1}(u) = \frac{1}{2}x\left( \nu xJ_{\nu}^2(x)-2\nu(\nu+1)J_{\nu}(x)J_{\nu+1}(x)+(\nu+1)xJ_{\nu+1}^2(x) \right) ,
\\
   \int^x_0 du uJ_{\nu}^2(u) = x\left( xJ_{\nu}^2(x)-2\nu J_{\nu}(x)J_{\nu+1}(x)+xJ_{\nu+1}^2(x) \right) ,
\end{gather*}
and the Neumann expansions
\begin{gather*}
   x_{j} = i\phi_{j}(s) + i\int^{s}_0 dz K_{1}(s,z)\phi_{j}(z) + i\int^{s}_0 dz \int^{s}_0 dz' K_{1}(s,z) K_{1}(z,z')\phi_{j}(z') + \ldots
\\
   y_{j} = i\psi_{j}(s) + i\int^{s}_0 dz K_{1}(z,s)\psi_{j}(z) + i\int^{s}_0 dz \int^{s}_0 dz' K_{1}(z',z) K_{1}(z,s)\psi_{j}(z') + \ldots 
\end{gather*}
we can deduce the behaviour of the variables in the neighbourhood of $ s=0 $, which furnishes the initial conditions
for the integrals of motion to be deduced below,
\begin{align}
  x_{0}(s) & \mathop{\sim}\limits_{s \to 0} i \phi_0(s) \sim i\frac{s^{-\nu_0}}{\Gamma(\nu_1-\nu_0+1)} ,
\\    
  x_{1}(s) & \mathop{\sim}\limits_{s \to 0} i \phi_1(s) \sim i\frac{\nu_0s^{-\nu_0}}{\Gamma(\nu_1-\nu_0+1)}+i\frac{(1-\nu_0)s^{-\nu_0+1}}{\Gamma(\nu_1-\nu_0+2)} ,
\\    
  y_{0}(s) & \mathop{\sim}\limits_{s \to 0} i \psi_0(s) \sim -i\frac{\nu_0s^{\nu_1}}{\Gamma(\nu_1-\nu_0+1)}+i\frac{(\nu_0-1)s^{\nu_1+1}}{\Gamma(\nu_1-\nu_0+2)} ,
\\    
  y_{1}(s) & \mathop{\sim}\limits_{s \to 0} i \psi_1(s) \sim i\frac{s^{\nu_1}}{\Gamma(\nu_1-\nu_0+1)} ,
\\    
  \eta_{0}(s) & \mathop{\sim}\limits_{s \to 0} -\frac{s^{\nu_1-\nu_0+1}}{\Gamma(\nu_1-\nu_0+2)\Gamma(\nu_1-\nu_0+1)} , 
\label{M=1BC} \\    
  \eta_{1}(s) & \mathop{\sim}\limits_{s \to 0} -\frac{\nu_0s^{\nu_1-\nu_0+1}}{\Gamma(\nu_1-\nu_0+2)\Gamma(\nu_1-\nu_0+1)}
                                                 +\frac{(1-2\nu_0)s^{\nu_1-\nu_0+2}}{\Gamma(\nu_1-\nu_0+3)\Gamma(\nu_1-\nu_0+1)} , 
\\
  \xi_{0}(s) & \mathop{\sim}\limits_{s \to 0} \nu_{0}\nu_{1}+\frac{\nu_0(\nu_1-\nu_0+1)s^{\nu_1-\nu_0+1}}{\Gamma^2(\nu_1-\nu_0+2)}
                                                            +\frac{(1-2\nu_0)s^{\nu_1-\nu_0+2}}{\Gamma(\nu_1-\nu_0+3)\Gamma(\nu_1-\nu_0+1)} , 
\\
  \xi_{1}(s) & \mathop{\sim}\limits_{s \to 0} -\nu_0-\nu_{1}-\frac{s^{\nu_1-\nu_0+1}}{\Gamma(\nu_1-\nu_0+2)\Gamma(\nu_1-\nu_0+1)} .
\end{align}
Consequently we note an analogue of the orthogonality relation given in Remark \ref{M=1split}
\begin{equation*}
  \frac{y_{1}}{x_{0}} \mathop{\sim}\limits_{s \to 0} s^{\nu_0+\nu_1} ,\quad
  \frac{x_{1}}{y_{0}} \mathop{\sim}\limits_{s \to 0} -s^{-\nu_0-\nu_1} .
\end{equation*}

As was the case in the Tracy and Widom theory we would like to reduce the order of the coupled ODE system and deduce
the first integrals of the motion. For convenience we define the elementary symmetric functions $ e_j, j=1,2 $ of $ \nu_0,\nu_1 $.
\begin{proposition}\label{M=1Integrals}
The system possesses the integrals of motion
\begin{equation}
  \xi_{1} = \eta_{0}-e_1 ,\quad {\rm Tr}C = -e_1 ;
\label{1stIntegral}
\end{equation}
the orthogonality relation
\begin{equation}
   {\rm Tr}A^{(2)} = x_{0}y_{0}+x_{1}y_{1} = 0 ;
\label{3rdIntegral}
\end{equation}
the further integrals of motion 
\begin{equation}
  \eta_{1} + \xi_{0} = e_2 ;
\label{2ndIntegral}
\end{equation}
\begin{equation}
   sx_{0}y_{1} = -\eta_{0}\xi_1+\eta_{0}+\xi_{0}-\eta_1-e_2 ;
\label{4thIntegral}
\end{equation} 
and with $\eta_0$ identified as the Hamiltonian, the identity
\begin{equation}
   \eta_{0}x_{0}y_{0}+(\eta_{1}-\xi_{0}-s)x_{0}y_{1}+x_{1}y_{0}-\xi_{1}x_{1}y_{1}+\eta_{0} = 0.
\label{Energy}
\end{equation} 
\end{proposition}
\begin{proof}
Subtracting \eqref{M=1ODE:7} from \eqref{M=1ODE:7}
\begin{equation*}
  \xi_{1}'-\eta_{0}' = 0,
\end{equation*}
therefore
\begin{equation*}
  \xi_{1}-\eta_{0} = -e_1 .
\end{equation*}
and \eqref{1stIntegral} follows.

Adding $ y_{0}\times $\eqref{M=1ODE:1} to $ x_{0}\times$\eqref{M=1ODE:4} we find
\begin{equation*}
   s(x_{0}y_{0})' = (\eta_{1}-\xi_{0}-s)x_{0}y_{1}-x_{1}y_{0} . 
\end{equation*} 
On the other hand adding $ y_{1}\times $\eqref{M=1ODE:2} to $ x_{1}\times$\eqref{M=1ODE:3} we deduce
\begin{equation*}
   s(x_{1}y_{1})' = (\xi_{0}-\eta_{1}+s)x_{0}y_{1}+x_{1}y_{0} . 
\end{equation*} 
Thus we find their sum vanishes and for $ s \neq 0 $
\begin{equation*}
   (x_{0}y_{0})'+(x_{1}y_{1})' = 0 ,
\end{equation*}
and application of the initial values gives \eqref{3rdIntegral}. An immediate consequence of this latter relation
and \eqref{M=1ODE:5} and \eqref{M=1ODE:8} is
\begin{equation*}
   \xi_{0}'+\eta_{1}' = 0 .
\end{equation*} 
Applying the values at $ s= 0 $ we conclude that \eqref{2ndIntegral} is satisfied.

Adding $ y_{1}\times $\eqref{M=1ODE:1} and $ x_{0}\times$\eqref{M=1ODE:3} and then employing \eqref{M=1ODE:5} to \eqref{M=1ODE:8}
we find 
\begin{equation*}
   s(x_{0}y_{1})' = -\xi_{1}\eta_{0}'-\eta_{0}\xi_{1}'+\xi_{0}'-\eta_{1}' .
\end{equation*}
Utilising \eqref{M=1ODE:7} once again we have the total derivative
\begin{equation*}
    (sx_{0}y_{1})'-\eta_{0}' = -\xi_{1}\eta_{0}'-\eta_{0}\xi_{1}'+\xi_{0}'-\eta_{1}' .
\end{equation*}
Employing \eqref{1stIntegral} and \eqref{2ndIntegral} we have \eqref{4thIntegral}.

Forming $ x_{0}'\times$\eqref{M=1ODE:4} minus $ y_{0}'\times $\eqref{M=1ODE:1} and simplifying we arrive at
\begin{equation*}
   0 = (\eta_{1}-\xi_{0}-s)x_{0}'y_{1}+\eta_{0}(x_{0}y_{0})'+x_{1}y_{0}' .
\end{equation*}
Next forming $ x_{1}'\times$\eqref{M=1ODE:3} minus $ y_{1}'\times $\eqref{M=1ODE:2} and simplifying we have
\begin{equation*}
   0 = (\eta_{1}-\xi_{0}-s)x_{0}y_{1}'-\xi_{1}(x_{1}y_{1})'+x_{1}'y_{0} .
\end{equation*}
Adding these two later relations we compute
\begin{align*}
  0 = & (\eta_{1}-\xi_{0}-s)(x_{0}y_{1})'+\eta_{0}(x_{0}y_{0})'-\xi_{1}(x_{1}y_{1})'+(x_{1}y_{0})'
\nonumber \\
    = & \left[(\eta_{1}-\xi_{0}-s)x_{0}y_{1}\right]'-(\eta_{1}'-\xi_{0}'-1)x_{0}y_{1}+(\eta_{0}x_{0}y_{0})'-\eta_{0}'x_{0}y_{0}-(\xi_{1}x_{1}y_{1})'+\xi_{1}'x_{1}y_{1}+(x_{1}y_{0})'
\nonumber \\
    = & \left[ (\eta_{1}-\xi_{0}-s)x_{0}y_{1}+\eta_{0}x_{0}y_{0}-\xi_{1}x_{1}y_{1}+x_{1}y_{0} \right]'+x_{0}y_{1}-(x_{1}y_{1}-x_{0}y_{0})x_{0}y_{1}-x_{0}y_{1}x_{0}y_{0}+x_{0}y_{1}x_{1}y_{1}
\nonumber \\
    = & \left[ (\eta_{1}-\xi_{0}-s)x_{0}y_{1}+\eta_{0}x_{0}y_{0}-\xi_{1}x_{1}y_{1}+x_{1}y_{0} \right]'+x_{0}y_{1}
\nonumber \\
    = & \left[ (\eta_{1}-\xi_{0}-s)x_{0}y_{1}+\eta_{0}x_{0}y_{0}-\xi_{1}x_{1}y_{1}+x_{1}y_{0}+\eta_{0} \right]' .
\nonumber
\end{align*}
Appealing to the initial conditions we deduce \eqref{Energy}, and consequently $ H=\eta_{0} $.
\end{proof}

Another feature of the Tracy and Widom theory is the appearance of the $\sigma$-forms for the resolvent function 
(for justification of this terminology, see \cite[Section \S 9.3]{For_2010})
which is also easily deduced in the generalised theory. 
\begin{proposition}
The resolvent function $ \eta_0(s) $ (recall \eqref{M=1tau}) satisfies a specialised $\sigma$-form equation for Painlev\'e III'
\begin{equation}
    s^2(\eta_0^{\prime\prime})^2 - e_1^2(\eta_0^{\prime})^2 + 4(\eta_0^{\prime})^2\left( s\eta_0^{\prime}-\eta_0+s+e_2 \right) - 4\eta_0\eta_0^{\prime} = 0 ,
\label{M=1sigma}
\end{equation}
subject to the boundary conditions \eqref{M=1BC} at $ s=0 $. The resolvent is given in terms of Okamoto's function $ h(s;v_1,v_2) $
(see Prop. 4.1 of \cite{FW_2002a}, or Eq. (0.7) of \cite{Oka_1987b})
\begin{equation*}
  \eta_0(s) = h(s)-\frac{s}{2}-\frac{1}{4}(\nu_1-\nu_0)^2 ,
\end{equation*}
for the special case $ v_1=v_2 = \pm(\nu_1-\nu_0) $.
\end{proposition}
\begin{proof}
We follow the Okamoto prescription and recast the dynamical variables in terms of $ \eta_0 $ and its derivatives.
From \eqref{M=1ODE:7} we have $ \eta_0^{\prime\prime} = x_0^{\prime}y_1+x_0y_1^{\prime} $. On the other hand we
deduce from \eqref{3rdIntegral} and the formulae for $ y_0 $ and $ x_1 $ using \eqref{M=1ODE:3} and \eqref{M=1ODE:1}
respectively that $ sx_0y_1^{\prime}-sx_0^{\prime}y_1 = e_1\eta_0^{\prime} $. Combining these two relations we have
\begin{equation}
   sx_0y_1^{\prime} = \frac{1}{2}\left( s\eta_0^{\prime\prime}+e_1\eta_0^{\prime} \right), \quad
   sx_0^{\prime}y_1 = \frac{1}{2}\left( s\eta_0^{\prime\prime}-e_1\eta_0^{\prime} \right) .
\label{aux}
\end{equation}
Now we use the same relations to eliminate $ y_0 $ and $ x_1 $ in the energy conservation relation \eqref{Energy}
and we find
\begin{equation*}
   0 = \eta_0 + (\eta_0-e_2-s-sx_0y_1)x_0y_1 - s^2x_0^{\prime}y_1^{\prime} ,
\end{equation*} 
where we have utilised \eqref{4thIntegral} in the last step. Finally using the identity
\begin{equation*}
   x_0^{\prime}y_1^{\prime} = \frac{(sx_0^{\prime}y_1)(sx_0y_1^{\prime})}{s^2x_0y_1} ,
\end{equation*}
and substituting for the $ \eta_0 $ derivatives we arrive at \eqref{M=1sigma}.
\end{proof}

\begin{remark}\label{M=1Ham}
The Hamiltonian variables can then be computed in terms of the resolvent and are given by
\begin{gather*}
  x_0^2 = s^{-\nu_0-\nu_1}\eta_0^{\prime}, \quad y_1^2 = s^{\nu_0+\nu_1}\eta_0^{\prime} ,
\\
  x_1 = -x_0 \left[ \frac{1}{2}s\frac{\eta_0^{\prime\prime}}{\eta_0^{\prime}}+\eta_0-\frac{1}{2}(\nu_0+\nu_1) \right] ,
\\
  y_0 = y_1 \left[ \frac{1}{2}s\frac{\eta_0^{\prime\prime}}{\eta_0^{\prime}}+\eta_0-\frac{1}{2}(\nu_0+\nu_1) \right] ,
\\
  \xi_0 = \nu_0\nu_1 - \frac{1}{2}\left[ -s\eta_0^{\prime}+\eta_0(1+\nu_0+\nu_1-\eta_0) \right] ,
\\
  \xi_1 = \eta_0-\nu_0-\nu_1, \quad \eta_1 = \frac{1}{2}\left[ -s\eta_0^{\prime}+\eta_0(1+\nu_0+\nu_1-\eta_0) \right] .
\end{gather*}
The relations for $ \xi_0 $ and $ \eta_1 $ follow from \eqref{2ndIntegral} and \eqref{4thIntegral}.
\end{remark}

We now give relations between the two sets $ (x_0,y_0) $ and $ (x_1,y_1) $ which we call {\it folding relations} 
and the proof of these.
\begin{proposition}\label{M=1Fold}
Assume $ x_{0} \neq 0 $. Then $ x_{1}, y_{1} $ are related to $ x_{0}, y_{0} $ by
\begin{equation}
   x_{1} = -s^{e_1}y_{0}, \quad y_{1} = s^{e_1}x_{0} .
\label{M=1fold}
\end{equation} 
\end{proposition}
\begin{proof}
Let $ y_{1} = f(s)x_{0} $, so that $ x_{1} = -f^{-1}(s)y_{0} $ using \eqref{3rdIntegral}. 
Substituting this into \eqref{M=1ODE:3} we have
\begin{equation*}
    sx_{0}f' = f\left( -sx_{0}'-\xi_{1}x_{0}+f^{-1}y_{0} \right) .
\end{equation*}
Now employing \eqref{M=1ODE:1} into the right-hand side of the above relation we deduce
\begin{equation*}
   sx_{0} \frac{f'}{f} = (\eta_{0}-\xi_{1})x_{0} = e_1 x_{0} .
\end{equation*} 
Under our assumption $ x_{0} \neq 0 $ we then have $ sf' = e_1f $ which has the general solution $ f=s^{e_1} $ given the
initial condition $ y_{1}/x_{0} \to s^{e_1} $ as $ s \to 0 $. These relations also follow easily from the relations of
Remark \ref{M=1Ham}.
\end{proof}

\begin{remark}
The relations \eqref{M=1fold} are the non-linear analogues of the relations given in Remark \ref{M=1split} for the
kernel functions. They also correct in an essential way assertions made in Remark (c) on pg. 9 of \cite{Str_2014}.
Understanding these relations for $ M=1 $ is key to that of the more general case $ M>1 $. 
\end{remark}

Having reduced our system to the pair of canonical variables $ (x_0,y_0) $ we are at the stage of discussing co-incidence with 
the original theory of Tracy and Widom \cite{TW_1994a}. For convenience we will set $ \nu_0=0 $ in this discussion.
\begin{lemma}
Let $ \nu_0=0 $ and $ \nu_1=\nu $. Expressed solely in terms of $ x_{0}, y_{0} $ the equations of motion are
\begin{gather}
   sx_{0}' = -\eta_{0}x_{0}+s^{-\nu}y_{0} ,
\label{M=1fODE:1} \\
   sy_{0}' = -(2\xi_{0}+s)s^{\nu}x_{0}+\eta_{0}y_{0} ,
\label{M=1fODE:2} \\
   \eta_{0}' = s^{\nu}x_{0}^2 ,
\label{M=1fODE:3} \\
   \xi_{0}' = x_{0}y_{0} .
\label{M=1fODE:4} 
\end{gather}
Equation \eqref{4thIntegral} is now
\begin{equation}
   s^{\nu+1}x_{0}^2 = 2\xi_{0}+\eta_{0}-\eta_{0}(\eta_{0}-\nu) ,
\label{M=1fold4th}
\end{equation} 
and \eqref{Energy} is
\begin{equation}
   -s^{-\nu}y_{0}^2-(2\xi_{0}+s)s^{\nu}x_{0}^2+(2\eta_{0}-\nu)x_{0}y_{0}+\eta_{0} = 0 .
\label{M=1foldEnergy}
\end{equation} 
\end{lemma}
\begin{proof}
Using \eqref{M=1fold} both \eqref{M=1ODE:1} and \eqref{M=1ODE:3} reduce to \eqref{M=1fODE:1}, while \eqref{M=1ODE:2} and
\eqref{M=1ODE:4} reduce to \eqref{M=1fODE:2}.
\end{proof}

\begin{proposition}
The current system $ \nu, s, x_{0}, y_{0}, \eta_{0}, \xi_{0} $
maps to that of Tracy and Widom \cite{TW_1994a} $ \alpha, t, q(t), p(t), u(t), v(t) $ under the transformations
$ \nu = \alpha $, $ s = \frac{1}{4}t $
\begin{gather*}
   x_{0}(s) = is^{-\nu/2}q(t) ,
\qquad
   y_{0}(s) = is^{\nu/2}\left( p(t)-\frac{\alpha}{2}q(t) \right) ,
\\
   \eta_{0}(s) = -\tfrac{1}{4}u(t) ,
\qquad
   \xi_{0}(s) = \tfrac{1}{4}\left( -v(t)+\frac{\alpha}{2}u(t) \right) .
\end{gather*}
\end{proposition}
\begin{proof}
We proceed by way of verification from our own results by direct calculation. Thus we find \eqref{M=1fold4th} becomes
\begin{equation*}
  tq^2 = \tfrac{1}{4}u^2+u+2v ,
\end{equation*}
i.e. Eq. (2.19) of \cite{TW_1994a}; \eqref{M=1foldEnergy} becomes
\begin{equation*}
   u = 4p^2-(\alpha^2-t+2v)q^2+2qpu ,
\end{equation*} 
which is Eq. (2.20) of \cite{TW_1994a}; \eqref{M=1fODE:3} and \eqref{M=1fODE:4} become
\begin{equation*}
   \frac{d}{dt}u = q^2 , \quad \frac{d}{dt}v = qp ,
\end{equation*}
which are Eqs. (2.24) and (2.25) of \cite{TW_1994a} respectively; and finally \eqref{M=1fODE:1} and \eqref{M=1fODE:2} become
\begin{equation*}
   t\frac{d}{dt}q = p+\tfrac{1}{4}qu ,
\qquad
   t\frac{d}{dt}p = (\tfrac{1}{4}\alpha^2-\tfrac{1}{4}t+\tfrac{1}{2}v)q-\tfrac{1}{4}pu ,
\end{equation*}
which match Eqs. (2.22) and (2.23) of \cite{TW_1994a} respectively.
\end{proof}

\section{\texorpdfstring{$ M = 2 $}{2ND} Theory at the Hard Edge}\label{M=2theory}
\setcounter{equation}{0}
\subsection{Fredholm Theory}

Here we treat the case $M=2$ with a single interval $J=(0,s)$ and thus $L=1$. As before we define 
the elementary symmetric functions $ e_j, j=1,2,3 $ of $ \nu_0,\nu_1,\nu_2 $. In this case application of
Prop. 3.9 of \cite{Str_2014} for $ J=(0,s) $, $ a_1=0, a_2=s $, i.e $ x_{j} = x^{(2)}_{j} $, $ y_{j} = y^{(2)}_{j} $
yields the following system of coupled ODEs
\begin{align}
  s x_{0}' & = -\eta_{0}x_{0}-x_{1} ,
\label{M=2ODE:1} \\
  s x_{1}' & = -\eta_{1}x_{0}-x_{2} ,
\label{M=2ODE:2} \\
  s x_{2}' & = -\eta_{2}x_{0}-sx_{0}+\xi_{0}x_{0}+\xi_{1}x_{1}+\xi_{2}x_{2} ,
\label{M=2ODE:3} \\
  s y_{2}' & = -\xi_{2}y_{2}+y_{1} ,
\label{M=2ODE:4} \\
  s y_{1}' & = -\xi_{1}y_{2}+y_{0} ,
\label{M=2ODE:5} \\
  s y_{0}' & = -\xi_{0}y_{2}+sy_{2}+\eta_{0}y_{0}+\eta_{1}y_{1}+\eta_{2}y_{2} ,
\label{M=2ODE:6} \\
  \xi_{0}' & = -x_{0}y_{0} ,
\label{M=2ODE:7} \\
  \xi_{1}' & = -x_{0}y_{1} ,
\label{M=2ODE:8} \\
  \xi_{2}' & = -x_{0}y_{2} ,
\label{M=2ODE:9} \\
  \eta_{0}' & = -x_{0}y_{2} ,
\label{M=2ODE:10} \\
  \eta_{1}' & = -x_{1}y_{2} ,
\label{M=2ODE:11} \\
  \eta_{2}' & = -x_{2}y_{2} .
\label{M=2ODE:12}
\end{align}
The Hamiltonian \eqref{M_Ham} is now
\begin{equation}
  H = -\eta_{0}x_{0}y_{0}-\eta_{1}x_{0}y_{1}+(\xi_{0}-\eta_{2}-s)x_{0}y_{2}-x_{1}y_{0}-x_{2}y_{1}+\xi_{1}x_{1}y_{2}+\xi_{2}x_{2}y_{2} ,
\end{equation}
and as before the Hamiltonian equations of motion
\begin{equation*}
  s x_{j}' = \frac{\partial}{\partial y_{j}}H, \quad  s y_{j}' = -\frac{\partial}{\partial x_{j}}H, \quad j=0,1,2 ,
\end{equation*}
\begin{equation*}
  \eta_{j}' = \frac{\partial}{\partial \xi_{j}}H, \quad  \xi_{j}' = -\frac{\partial}{\partial \eta_{j}}H, \quad j=0,1,2 ,
\end{equation*}
give rise to the previous set of equations \eqref{M=2ODE:1}-\eqref{M=2ODE:12}.

In the matrix formulation of the isomonodromy deformation problem we recall the definitions
\begin{equation*}
   E \coloneqq \begin{pmatrix} 0 & 0 & 0 \\ 0 & 0 & 0 \\ -1 & 0 & 0 \end{pmatrix}, \quad
   C \coloneqq \begin{pmatrix} -\eta_{0} & -1 & 0 \\ -\eta_{1} & 0 & -1 \\ \xi_{0}-\eta_{2} & \xi_{1} & \xi_{2} \end{pmatrix} ,
\end{equation*}
and
\begin{equation*}
   A \coloneqq A^{(2)} = \begin{pmatrix} x_{0}y_{0} & x_{0}y_{1} & x_{0}y_{2} \\ x_{1}y_{0} & x_{1}y_{1} & x_{1}y_{2} \\ x_{2}y_{0} & x_{2}y_{1} & x_{2}y_{2} \end{pmatrix}
     = \begin{pmatrix} x_{0} \\ x_{1} \\ x_{2} \end{pmatrix} \otimes \begin{pmatrix}y_{0} & y_{1} & y_{2} \end{pmatrix} . 
\end{equation*}
Again $ A^{(2)} $ is a rank $ 1 $ matrix so $ \det A^{(2)} = 0 $. The Schlesinger equations take the standard form
\begin{gather}
s A^{(2)\prime} = \left[ C+sE, A^{(2)} \right] ,
\label{schlesinger:1}\\
C^{\prime} = \left[ E, A^{(2)} \right] .
\label{schlesinger:2}
\end{gather}

\begin{proposition}
For $ M=2 $ the isomonodromic system \eqref{1st_Lax} and \eqref{2nd_Lax} has the singularity pattern $ \frac{4}{3}\!+\!1\!+\!1 $
and its Riemann-Papperitz symbol is
\begin{equation}
   \left\{ \begin{array}{cccc}
            0 & 1 & \multicolumn{2}{c}{\infty(\frac{1}{3})} \\
            -\nu_0 & 0 & s^{1/3} & -\frac{2}{3} \\
            -\nu_1 & 0 & \omega s^{1/3} & -1 \\
            -\nu_2 & 0 & \omega^2 s^{1/3} & -\frac{1}{3}+\nu_0+\nu_1+\nu_2
           \end{array}
   \right\} , \qquad \omega^3 = 1 .
\label{R-PsymbolM=2}
\end{equation} 
\end{proposition}
\begin{proof}
After mapping $ z \mapsto sz $ and $ \Psi(sz,s) \mapsto \Psi(z,s) $ the isomonodromic system become
\begin{equation}
\frac{\partial}{\partial z}\Psi = \left\{ sE + \frac{C-A^{(2)}}{z} + \frac{A^{(2)}}{z-1} \right\}\Psi ,
\label{iso_M=2}
\end{equation} 
and
\begin{equation*}
\frac{\partial}{\partial s}\Psi = \left\{ s^{-1}E z+s^{-1}C \right\}\Psi .
\end{equation*} 
$ E $ is nilpotent with eigenvalues $ 0, 0, 0 $ in Jordan blocks of size 2 \& 1, i.e. the resonant or ramified case;
the eigenvalues of $ C-A^{(2)} $ are $ -\nu_0, -\nu_1, -\nu_2 $ whilst those of $ A^{(2)} $ are $ 0, 0, 0 $.	
Again let us denote the matrix in braces on the right-hand side of \eqref{iso_M=2} by $ A $. The Jordan decomposition of $ sE $ is
\begin{equation*}
     sE = \begin{pmatrix} 0 & -s^{-1} & 0 \\ 0 & 0 & 1 \\ 1 & 0 & 0 \end{pmatrix}\cdot
          \begin{pmatrix} 0 & 1 & 0 \\ 0 & 0 & 0 \\ 0 & 0 & 0 \end{pmatrix}\cdot
          \begin{pmatrix} 0 & -s^{-1} & 0 \\ 0 & 0 & 1 \\ 1 & 0 & 0 \end{pmatrix}^{-1} ,
\end{equation*}
so we transform the system \eqref{iso_M=2} to 
$ B = \begin{pmatrix} 0 & -s^{-1} & 0 \\ 0 & 0 & 1 \\ 1 & 0 & 0 \end{pmatrix}^{-1}\cdot A
 \cdot\begin{pmatrix} 0 & -s^{-1} & 0 \\ 0 & 0 & 1 \\ 1 & 0 & 0 \end{pmatrix} $.
We now apply the first shearing transformation $ S \coloneqq {\rm diag}(1,z^{-g},z^{-2g}) $ with an arbitrary exponent 
$ g $ and form the new coefficient matrix $ C = S^{-1}\cdot B\cdot S - S^{-1}\cdot S' $. The leading exponent matrix 
of $ C $ is $ \begin{pmatrix} 1 & g & 2g+1 \\ -g+2 & 1 & g+1 \\ -2g+1 & -g+1 & 1 \end{pmatrix} $. The smallest 
positive exponent that allows us to have off-diagonal balance occurs when $ -2g+1=g $, i.e. if we choose $ g=1/3 $.
Under this choice the leading coefficient of $ C $ (which appears at order $ z^{-1/3} $) is now 
\begin{equation*}
   \begin{pmatrix} 0 & 1 & 0 \\ 0 & 0 & 0 \\ -1 & 0 & 0 \end{pmatrix} .
\end{equation*}
A Jordan decomposition of this reveals
\begin{equation*}
   \begin{pmatrix} 0 & 1 & 0 \\ 0 & 0 & 0 \\ -1 & 0 & 0 \end{pmatrix} =
   \begin{pmatrix} 0 & -1 & 0 \\ 0 & 0 & -1 \\ 1 & 0 & 0 \end{pmatrix} \cdot
   \begin{pmatrix} 0 & 1 & 0 \\ 0 & 0 & 1 \\ 0 & 0 & 0 \end{pmatrix} \cdot
   \begin{pmatrix} 0 & -1 & 0 \\ 0 & 0 & -1 \\ 1 & 0 & 0 \end{pmatrix}^{-1} ,
\end{equation*}
and we now have a $ 3\times 3 $ Jordan block. We apply this decomposition transformation and define another 
coefficient matrix 
$ D = \begin{pmatrix} 0 & -1 & 0 \\ 0 & 0 & -1 \\ 1 & 0 & 0 \end{pmatrix}^{-1}\cdot C
      \cdot\begin{pmatrix} 0 & -1 & 0 \\ 0 & 0 & -1 \\ 1 & 0 & 0 \end{pmatrix} $. 
In addition we define a new spectral variable $ z=aw^3 $ because of the fractional exponent for $ g $ and $a$
is a constant to be fixed later. Next we perform a large $ w $ expansion of $ D $
\begin{equation*}
   3aw^2 D = 
   3a^{2/3}w\begin{pmatrix} 0 & 1 & 0 \\ 0 & 0 & 1 \\ 0 & 0 & 0 \end{pmatrix}
   -3a^{1/3}s\eta_1\begin{pmatrix} 0 & 0 & 1 \\ 0 & 0 & 0 \\ 0 & 0 & 0 \end{pmatrix}
   +{\rm O}(w^{-1}) .
\end{equation*}
We now apply a second shearing transformation $ T \coloneqq {\rm diag}(1,w^{-h},w^{-2h}) $ with another arbitrary 
exponent $ h $ and form the new coefficient matrix $ G = T^{-1}\cdot D\cdot T - T^{-1}\cdot T' $. Now the leading 
exponent matrix of $ G $ is $ \begin{pmatrix} 1 & h-1 & 2h \\ -h+3 & 1 & h-1 \\ -2h+2 & -h+3 & 1 \end{pmatrix} $. 
The smallest positive exponent that allows us to have off-diagonal balance arises when $ -2h+2=h-1 $, i.e. if we 
choose $ h=1 $. The integer exponent signals the end of the recursive process of shearing transformations and a
diagonalisable matrix. With this value of $ h $ we find an expansion for $ G $ as $ w \to \infty $
\begin{equation*}
   G = 
    \begin{pmatrix} 0 & 3a^{2/3} & 0 \\ 0 & 0 & 3a^{2/3} \\ -3a^{-1/3}s & 0 & 0 \end{pmatrix}
   +w^{-1}\begin{pmatrix} 2 & 0 & 0 \\ 0 & 1+3\xi_2 & 0 \\ 0 & 0 & 3-3\eta_0 \end{pmatrix}
   +{\rm O}(w^{-2}) .
\end{equation*}
The leading order matrix appearing above is now diagonalisable, and by choosing $ a=-1/27 $, we can simplify the
decomposition in order that the final transformed coefficient matrix $ H $ has the expansion as $ w \to \infty $
\begin{multline*}
   H =
   \begin{pmatrix} s^{1/3} & 0 & 0 \\ 0 & \omega s^{1/3} & 0 \\ 0 & 0 & \omega^2 s^{1/3} \end{pmatrix} 
\\
  +w^{-1}\begin{pmatrix} 2-e_1 & \omega^2\xi_2-\eta_0+(1-\omega^2)/3 & \omega\xi_2-\eta_0+(1-\omega)/3 \\
                         \omega\xi_2-\eta_0+(1-\omega)/3 & 2-e_1 & \omega^2\xi_2-\eta_0+(1-\omega^2)/3 \\
                         \omega^2\xi_2-\eta_0+(1-\omega^2)/3 & \omega\xi_2-\eta_0+(1-\omega)/3 & 2-e_1 \end{pmatrix}
  +{\rm O}(w^{-2}) ,
\end{multline*}
where $ \omega $ is the third root of unity.
The sub-leading matrix appearing above can also be diagonalised as
\begin{equation*}
   \begin{pmatrix} \omega^2 & 1 & \omega \\ \omega & 1 & \omega^2 \\ 1 & 1 & 1 \end{pmatrix}\cdot
   \begin{pmatrix} 2 & 0 & 0 \\ 0 & 3-3\eta_0 & 0 \\ 0 & 0 & 1-3e_1+3\eta_0 \end{pmatrix}\cdot
   \begin{pmatrix} \omega^2 & 1 & \omega \\ \omega & 1 & \omega^2 \\ 1 & 1 & 1 \end{pmatrix}^{-1} ,
\end{equation*}
and these diagonal elements give us the last column of the Riemann-Papperitz symbol.
\end{proof}

\begin{definition}
We will define {\it generic conditions} to be $ \nu_2-\nu_1 \neq \Z $ whether or not $ \nu_0 $ is zero.
However from the random matrix application this is precisely the case of interest, and we observe that
this constraint can be lifted in principle with the proper treatment of logarithmic contributions. 
\end{definition}

For $M=2$, the kernel functions are given by
\begin{equation*}
   f(x) = G^{1,0}_{0,3}(x|-\nu_{0},-\nu_{1},-\nu_{2}) ,
\quad
   g(x) = G^{2,0}_{0,3}(x|\nu_{2},\nu_{1},\nu_{0}) ,
\end{equation*}
however we will employ hyper-Bessel functions representations, involving the generalised hypergeometric
function ${}_0 F_2$.
Using standard relations relating the Meijer G-function to the hyper-Bessel function, and 
differential-difference identities of the latter, and recalling the definitions \eqref{KernelFns}, we have
\begin{equation*}
   \phi_0(x) = -\frac{x^{-\nu_0}}{\Gamma \left(\nu_1-\nu_0+1\right) \Gamma \left(\nu_2-\nu_0+1\right)}{}_0F_2\left(;\nu_1-\nu_0+1,\nu_2-\nu_0+1;-x\right) ,
\end{equation*}
\begin{multline*}
   \phi_1(x) =
    -\frac{\nu_0 x^{-\nu_0}}{\Gamma \left(\nu_1-\nu_0+1\right) \Gamma \left(\nu_2-\nu_0+1\right)}{}_0F_2\left(;\nu_1-\nu_0+1,\nu_2-\nu_0+1;-x\right)
\\
    -\frac{x^{1-\nu_0}}{\Gamma \left(\nu_1-\nu_0+2\right) \Gamma \left(\nu_2-\nu_0+2\right)}{}_0F_2\left(;\nu_1-\nu_0+2,\nu_2-\nu_0+2;-x\right) ,              
\end{multline*}
\begin{multline*}
   \phi_2(x) =
   -\frac{\nu_0^2 x^{-\nu_0}}{\Gamma \left(\nu_1-\nu_0+1\right) \Gamma \left(\nu_2-\nu_0+1\right)}{}_0F_2\left(;\nu_1-\nu_0+1,\nu_2-\nu_0+1;-x\right)
\\
   +\frac{\left(1-2 \nu_0\right) x^{1-\nu_0}}{\Gamma \left(\nu_1-\nu_0+2\right) \Gamma \left(\nu_2-\nu_0+2\right)}{}_0F_2\left(;\nu_1-\nu_0+2,\nu_2-\nu_0+2;-x\right)
\\
   -\frac{x^{2-\nu_0}}{\Gamma \left(\nu_1-\nu_0+3\right) \Gamma \left(\nu_2-\nu_0+3\right)}{}_0F_2\left(;\nu_1-\nu_0+3,\nu_2-\nu_0+3;-x\right) ,
\end{multline*}
\begin{multline*}
   \psi_0(x) = 
   \frac{\Gamma \left(\nu_2-\nu_1\right) x^{\nu_1}}{\Gamma \left(\nu_1-\nu_0-1\right)}{}_0F_2\left(;\nu_1-\nu_0-1,\nu_1-\nu_2+1;x\right)
\\   
   +\frac{\Gamma \left(\nu_1-\nu_2\right) x^{\nu_2}}{\Gamma \left(\nu_2-\nu_0-1\right)}{}_0F_2\left(;\nu_2-\nu_0-1,\nu_2-\nu_1+1;x\right)
\\   
   +\left(\nu_0-\nu_1-\nu_2+1\right) \left[ \frac{\Gamma \left(\nu_2-\nu_1\right) x^{\nu_1}}{\Gamma \left(\nu_1-\nu_0\right)}{}_0F_2\left(;\nu_1-\nu_0,\nu_1-\nu_2+1;x\right)
    \right. 
\\ \left.
   +\frac{\Gamma \left(\nu_1-\nu_2\right) x^{\nu_2}}{\Gamma \left(\nu_2-\nu_0\right)}{}_0F_2\left(;\nu_2-\nu_0,\nu_2-\nu_1+1;x\right)
                                        \right]
\\   
   +\nu_1 \nu_2 \left[ \frac{\Gamma \left(\nu_2-\nu_1\right) x^{\nu_1}}{\Gamma \left(\nu_1-\nu_0+1\right)}{}_0F_2\left(;\nu_1-\nu_0+1,\nu_1-\nu_2+1;x\right)
    \right. 
\\ \left.
   +\frac{\Gamma \left(\nu_1-\nu_2\right) x^{\nu_2}}{\Gamma \left(\nu_2-\nu_0+1\right)}{}_0F_2\left(;\nu_2-\nu_0+1,\nu_2-\nu_1+1;x\right)
                  \right] ,   
\end{multline*}
\begin{multline*}
   \psi_1(x) =
   \frac{\Gamma \left(\nu_2-\nu_1\right) x^{\nu_1}}{\Gamma \left(\nu_1-\nu_0\right)}{}_0F_2\left(;\nu_1-\nu_0,\nu_1-\nu_2+1;x\right)
\\
   +\frac{\Gamma \left(\nu_1-\nu_2\right) x^{\nu_2}}{\Gamma \left(\nu_2-\nu_0\right)}{}_0F_2\left(;\nu_2-\nu_0,\nu_2-\nu_1+1;x\right)
\\
   -\left(\nu_1+\nu_2\right) \left[ \frac{\Gamma \left(\nu_2-\nu_1\right) x^{\nu_1}}{\Gamma \left(\nu_1-\nu_0+1\right)}{}_0F_2\left(;\nu_1-\nu_0+1,\nu_1-\nu_2+1;x\right)
   \right.
\\ \left.
   +\frac{\Gamma \left(\nu_1-\nu_2\right) x^{\nu_2}}{\Gamma \left(\nu_2-\nu_0+1\right)}{}_0F_2\left(;\nu_2-\nu_0+1,\nu_2-\nu_1+1;x\right)
                               \right] ,   
\end{multline*}
\begin{multline*}
    \psi_2(x) =
    \frac{\Gamma \left(\nu_2-\nu_1\right) x^{\nu_1}}{\Gamma \left(\nu_1-\nu_0+1\right)}{}_0F_2\left(;\nu_1-\nu_0+1,\nu_1-\nu_2+1;x\right)
\\
    +\frac{\Gamma \left(\nu_1-\nu_2\right) x^{\nu_2}}{\Gamma \left(\nu_2-\nu_0+1\right)}{}_0F_2\left(;\nu_2-\nu_0+1,\nu_2-\nu_1+1;x\right) .
\end{multline*} 
Here the linear orthogonality relation
\begin{equation*}
 \psi_0(x) \phi_0(x)+\psi_1(x) \phi_1(x)+\psi_2(x) \phi_2(x) = 0 ,
\end{equation*}
is not so obvious, and implies an bilinear identity involving hyper-Bessel functions with reflected arguments. 

The initial value conditions for the Hamiltonian variables can be imposed through an expansion in the 
neighbourhood of $ s = 0 $ with restricted argument. Thus we have
\begin{align}
   x_{0}(s) & \sim -\frac{is^{-\nu_0}}{\Gamma \left(\nu_1-\nu_0+1\right) \Gamma \left(\nu_2-\nu_0+1\right)} ,
\label{x0_IC}\\
   x_{1}(s) & \sim -\frac{i \nu_0 s^{-\nu_0}}{\Gamma \left(\nu_1-\nu_0+1\right) \Gamma \left(\nu_2-\nu_0+1\right)}
        -\frac{i \left(1-\nu_0\right) s^{1-\nu_0}}{\Gamma \left(\nu_1-\nu_0+2\right) \Gamma \left(\nu_2-\nu_0+2\right)} ,
\label{x1_IC}\\
   x_{2}(s) & \sim -\frac{i \nu_0^2 s^{-\nu_0}}{\Gamma \left(\nu_1-\nu_0+1\right) \Gamma \left(\nu_2-\nu_0+1\right)}
        +\frac{i \left(1-\nu_0\right)^2 s^{1-\nu_0}}{\Gamma \left(\nu_1-\nu_0+2\right) \Gamma \left(\nu_2-\nu_0+2\right)} ,
\label{x2_IC}
\end{align}
\begin{multline}
   y_{0}(s) \sim
         \frac{i \nu_0 \nu_2 \Gamma \left(\nu_2-\nu_1\right) s^{\nu_1}}{\Gamma \left(\nu_1-\nu_0+1\right)}
        -\frac{i \left(\nu_0\nu_2-\nu_0+\nu_1-\nu_2+1\right) \Gamma \left(\nu_2-\nu_1-1\right) s^{\nu_1+1}}{\Gamma \left(\nu_1-\nu_0+2\right)}
        \\
        +\frac{i \nu_0 \nu_1 \Gamma \left(\nu_1-\nu_2\right) s^{\nu_2}}{\Gamma \left(\nu_2-\nu_0+1\right)}
        -\frac{i \left(\nu_0\nu_1-\nu_0+\nu_2-\nu_1+1\right) \Gamma \left(\nu_1-\nu_2-1\right) s^{\nu_2+1}}{\Gamma \left(\nu_2-\nu_0+2\right)} ,
\label{y0_IC}
\end{multline}
\begin{equation}
   y_{1}(s) \sim 
         -\frac{i \left(\nu_0+\nu_2\right) \Gamma \left(\nu_2-\nu_1\right) s^{\nu_1}}{\Gamma \left(\nu_1-\nu_0+1\right)}
         -\frac{i \left(\nu_0+\nu_1\right) \Gamma \left(\nu_1-\nu_2\right) s^{\nu_2}}{\Gamma \left(\nu_2-\nu_0+1\right)},
\label{y1_IC}
\end{equation}
\begin{equation}
   y_{2}(s) \sim 
        \frac{i \Gamma \left(\nu_2-\nu_1\right) s^{\nu_1}}{\Gamma \left(\nu_1-\nu_0+1\right)}+\frac{i \Gamma \left(\nu_1-\nu_2\right) s^{\nu_2}}{\Gamma \left(\nu_2-\nu_0+1\right)} ,
\label{y2_IC}
\end{equation}
\begin{multline}
   \eta_{0}(s) \sim 
         -\frac{\Gamma \left(\nu_2-\nu_1\right) s^{\nu_1-\nu_0+1}}{\Gamma \left(\nu_1-\nu_0+2\right) \Gamma \left(\nu_1-\nu_0+1\right) \Gamma \left(\nu_2-\nu_0+1\right)}
          \\
         -\frac{\Gamma \left(\nu_1-\nu_2\right) s^{\nu_2-\nu_0+1}}{\Gamma \left(\nu_1-\nu_0+1\right) \Gamma \left(\nu_2-\nu_0+2\right) \Gamma \left(\nu_2-\nu_0+1\right)} ,
\label{M=2:eta0}
\end{multline}
\begin{multline*}
   \eta_{1}(s) \sim
        -\frac{\nu_0 \Gamma \left(\nu_2-\nu_1\right) s^{\nu_1-\nu_0+1}}{\Gamma \left(\nu_1-\nu_0+2\right) \Gamma \left(\nu_1-\nu_0+1\right) \Gamma \left(\nu_2-\nu_0+1\right)}
         \\
        +\frac{\left(-\nu_0^2-\nu_0\nu_1+2 \nu_0\nu_2+\nu_1-\nu_2+1\right) \Gamma \left(\nu_2-\nu_1-1\right) s^{\nu_1-\nu_0+2}}
              {\Gamma \left(\nu_1-\nu_0+3\right) \Gamma \left(\nu_1-\nu_0+1\right) \Gamma \left(\nu_2-\nu_0+2\right)}
        \\
        -\frac{\nu_0 \Gamma \left(\nu_1-\nu_2\right) s^{\nu_2-\nu_0+1}}{\Gamma \left(\nu_1-\nu_0+1\right) \Gamma \left(\nu_2-\nu_0+2\right) \Gamma \left(\nu_2-\nu_0+1\right)}
         \\
        +\frac{\left(-\nu_0^2-\nu_0\nu_2+2 \nu_0\nu_1+\nu_2-\nu_1+1\right) \Gamma \left(\nu_1-\nu_2-1\right) s^{\nu_2-\nu_0+2}}
              {\Gamma \left(\nu_1-\nu_0+2\right) \Gamma \left(\nu_2-\nu_0+3\right) \Gamma \left(\nu_2-\nu_0+1\right)} ,
\end{multline*}
\begin{multline*}
   \eta_{2}(s) \sim
      -\frac{\nu_0^2 \Gamma \left(\nu_2-\nu_1\right) s^{\nu_1-\nu_0+1}}{\Gamma \left(\nu_1-\nu_0+2\right) \Gamma \left(\nu_1-\nu_0+1\right) \Gamma \left(\nu_2-\nu_0+1\right)}
      \\
      -\frac{\left(\nu_0^3-2 \nu_0-\left(2 \nu_0^2-2 \nu_0+1\right) \nu_2+\left(1-\nu_0\right)^2 \nu_1+1\right) \Gamma \left(\nu_2-\nu_1-1\right) s^{\nu_1-\nu_0+2}}
           {\Gamma \left(\nu_1-\nu_0+3\right) \Gamma \left(\nu_1-\nu_0+1\right) \Gamma \left(\nu_2-\nu_0+2\right)}
      \\
      -\frac{\nu_0^2 \Gamma \left(\nu_1-\nu_2\right) s^{\nu_2-\nu_0+1}}{\Gamma \left(\nu_1-\nu_0+1\right) \Gamma \left(\nu_2-\nu_0+2\right) \Gamma \left(\nu_2-\nu_0+1\right)}
      \\
      -\frac{\left(\nu_0^3-2 \nu_0-\left(2 \nu_0^2-2 \nu_0+1\right) \nu_1+\left(1-\nu_0\right)^2 \nu_2+1\right) \Gamma \left(\nu_1-\nu_2-1\right) s^{\nu_2-\nu_0+2}}
            {\Gamma \left(\nu_1-\nu_0+2\right) \Gamma \left(\nu_2-\nu_0+3\right) \Gamma \left(\nu_2-\nu_0+1\right)} ,
\end{multline*}
\begin{multline}
   \xi_{0}(s) \sim -e_3
        -\frac{\nu_0 \nu_2 \Gamma \left(\nu_2-\nu_1\right) s^{\nu_1-\nu_0+1}}{\Gamma \left(\nu_1-\nu_0+2\right) \Gamma \left(\nu_1-\nu_0+1\right) \Gamma \left(\nu_2-\nu_0+1\right)}
        \\
        -\frac{\left(\nu_2 \nu_0^2-\nu_0^2-2 \nu_2^2 \nu_0+\nu_1 \nu_2 \nu_0+\nu_1 \nu_0+2 \nu_0+\nu_2^2-\nu_1 \nu_2-\nu_1-1\right) \Gamma \left(\nu_2-\nu_1-1\right) s^{\nu_1-\nu_0+2}}
              {\Gamma \left(\nu_1-\nu_0+3\right) \Gamma \left(\nu_1-\nu_0+1\right) \Gamma \left(\nu_2-\nu_0+2\right)}
        \\
        -\frac{\nu_0 \nu_1 \Gamma \left(\nu_1-\nu_2\right) s^{\nu_2-\nu_0+1}}{\Gamma \left(\nu_1-\nu_0+1\right) \Gamma \left(\nu_2-\nu_0+2\right) \Gamma \left(\nu_2-\nu_0+1\right)}
        \\
        -\frac{\left(\nu_1 \nu_0^2-\nu_0^2-2 \nu_1^2 \nu_0+\nu_1 \nu_2 \nu_0+\nu_2 \nu_0+2 \nu_0+\nu_1^2-\nu_1 \nu_2-\nu_2-1\right) \Gamma \left(\nu_1-\nu_2-1\right) s^{\nu_2-\nu_0+2}}
              {\Gamma \left(\nu_1-\nu_0+2\right) \Gamma \left(\nu_2-\nu_0+3\right) \Gamma \left(\nu_2-\nu_0+1\right)},
\label{M=2:xi0}
\end{multline}
\begin{multline}
   \xi_{1}(s) \sim  e_2
        +\frac{\left(\nu_0+\nu_2\right) \Gamma \left(\nu_2-\nu_1\right) s^{\nu_1-\nu_0+1}}
              {\Gamma \left(\nu_1-\nu_0+2\right) \Gamma \left(\nu_1-\nu_0+1\right) \Gamma \left(\nu_2-\nu_0+1\right)}
         \\
        +\frac{\left(\nu_0+\nu_1\right) \Gamma \left(\nu_1-\nu_2\right) s^{\nu_2-\nu_0+1}}
              {\Gamma \left(\nu_1-\nu_0+1\right) \Gamma \left(\nu_2-\nu_0+2\right) \Gamma \left(\nu_2-\nu_0+1\right)},
\label{M=2:xi1}
\end{multline}
\begin{multline}
   \xi_{2}(s) \sim  -e_1
         -\frac{\Gamma \left(\nu_2-\nu_1\right) s^{\nu_1-\nu_0+1}}{\Gamma \left(\nu_1-\nu_0+2\right) \Gamma \left(\nu_1-\nu_0+1\right) \Gamma \left(\nu_2-\nu_0+1\right)}
          \\
         -\frac{\Gamma \left(\nu_1-\nu_2\right) s^{\nu_2-\nu_0+1}}{\Gamma \left(\nu_1-\nu_0+1\right) \Gamma \left(\nu_2-\nu_0+2\right) \Gamma \left(\nu_2-\nu_0+1\right)} .
\label{M=2:xi2}
\end{multline}
Some warning ought to be attached to the above results. Only those terms where the $s$-exponent has the least (in sign)
real part should be admitted as the true lowest order term, depending on the relative sizes of the parameters. Whilst
the remaining terms still do contribute there will be additional, higher order terms arising from the leading one, and 
which will appear at the same order. However such higher order terms haven't been worked out in the above expressions. 

\subsection{First Integrals}
 
Again the system \eqref{M=2ODE:1}-\eqref{M=2ODE:12} can be reduced in order. This requires some preliminary results.
\begin{lemma}\label{M=2elim}
Eliminating the variables $ x_1, x_2, y_0, y_1 $ successively in favour of $ x_0, y_2 $ we have the relations
\begin{align}
  x_1 & = -\eta_0x_0-sx_0^{\prime} ,
\label{aux:x1}  \\
  x_2 & = -\eta_1x_0-sx_0^2y_2+(1+\eta_0)sx_0^{\prime}+s^2x_0^{\prime\prime} ,
\label{aux:x2}  \\
  y_1 & = \xi_2y_2+sy_2^{\prime} ,
\label{aux:y1}  \\
  y_0 & =  \xi_1y_2-sx_0y_2^2+(1+\xi_2)sy_2^{\prime}+s^2y_2^{\prime\prime} .
\label{aux:y0}
\end{align}
Consequently we have
\begin{align}
   \xi_1^{\prime} & = -sx_0y_2^{\prime}+(\eta_0-e_1)\eta_0^{\prime} ,
\label{aux:xi1}   \\
   \eta_1^{\prime} & = sx_0^{\prime}y_2-\eta_0\eta_0^{\prime} ,
\label{aux:eta1}   \\
   \xi_0^{\prime} & = -s^2x_0y_2^{\prime\prime}-(1-e_1+\eta_0)sx_0y_2^{\prime}+\eta_0^{\prime}\xi_1+s(\eta_0^{\prime})^2 ,
\label{aux:xi0}   \\
   \eta_2^{\prime} & = -s^2x_0^{\prime\prime}y_2-(1+\eta_0)sx_0^{\prime}y_2-\eta_0^{\prime}\eta_1+s(\eta_0^{\prime})^2 .
\label{aux:eta2}   
\end{align}
In addition we note
\begin{multline}
   s^3x_0^{\prime\prime\prime}+(e_1+3)s^2x_0^{\prime\prime}+(e_1+e_2+1-\eta_0-5sx_0y_2)sx_0^{\prime}-s^2x_0^2y_2^{\prime}
\\
   -(e_1+1)sx_0^2y_2-(\xi_0-\eta_2-\eta_0\xi_1-\xi_2\eta_1-s)x_0 = 0 ,
\label{M=2:third:1}
\end{multline}
\begin{multline}
   s^3y_2^{\prime\prime\prime}+(-e_1+3)s^2y_2^{\prime\prime}+(-e_1+e_2+1-\eta_0-5sx_0y_2)sy_2^{\prime}-s^2y_2^2x_0^{\prime}
\\
   -(-e_1+1)sx_0y_2^2+(\xi_0-\eta_2-\eta_0\xi_1-\xi_2\eta_1-s)y_2 = 0 .
\label{M=2:third:2}
\end{multline}
\end{lemma}

Now we have made sufficient preparation for the task of deducing the integrals of the motion.
\begin{proposition}\label{M=2integrals}
Let us assume generic conditions hold. 
For integral of motions we have the Hamiltonian giving the energy conservation
\begin{equation}
   \eta_{0}x_{0}y_{0}+\eta_{1}x_{0}y_{1}+(-\xi_{0}+\eta_{2}+s)x_{0}y_{2}+x_{1}y_{0}+x_{2}y_{1}-\xi_{1}x_{1}y_{2}-\xi_{2}x_{2}y_{2}+\eta_{0} = 0 ;
\label{M=2:Energy}
\end{equation}
the relations 
\begin{equation}
  \xi_{2} = \eta_{0} - e_1 ,\quad {\rm Tr}C = -e_1 ,
\label{M=2:1stIntegral}
\end{equation}
\begin{equation}
   sx_{0}y_{2} = \eta_{0}\xi_{2}+\eta_{1}-\xi_{1}+e_2-\eta_{0} ,
\label{M=2:4thIntegral}
\end{equation} 
and the orthogonality relation
\begin{equation}
   {\rm Tr}A^{(2)} = x_{0}y_{0}+x_{1}y_{1}+x_{2}y_{2} = 0 .
\label{M=2:3rdIntegral}
\end{equation}
In addition the latter relation can be integrated once again to give
\begin{multline}
  -3 e_3 +e_2 \left(e_1+\eta_0-1\right)-\eta_0 \left(e_1-\eta_0+1\right) \left(e_1+\eta_0-2\right)+\left(2 e_1-1\right) \eta_1+\left(1-e_1\right) \xi_1
   \\
   -s x_0 y_1+s x_0 y_2 \left(-2 \eta_0+\xi_2+2\right)+s x_1 y_2-3 \left(\eta_2+\xi_0\right) = 0 ,
\label{M=2:5thIntegral}
\end{multline}
and furthermore can be split into the two independent integrals
\begin{multline}
    3 e_3+e_2 \left(-2 e_1+\eta_0-4\right)+\eta_0 \left(e_1-\eta_0+1\right) \left(2 e_1-\eta_0+2\right)+\left(-e_1-1\right) \eta_1+ \left(2 e_1-3 \eta_0+4\right)\xi_1
   \\
   +2 s x_0  y_1+s x_0 y_2 \left(2 e_1-\eta_0+2\right)+s x_1 y_2+3 \xi_0 = 0 ,
\label{M=2:6thIntegral}
\end{multline}
\begin{multline}
   e_2 \left(e_1+\eta_0-1\right)-\eta_0 \left(e_1-\eta_0+1\right) \left(e_1+\eta_0-2\right)+ \left(-e_1+3 \eta_0-4\right)\eta_1+\left(1-e_1\right) \xi_1
   \\
   -s x_0 y_1-s x_0 y_2 \left(e_1+\eta_0-2\right)-2 s x_1 y_2 +3 \eta_2 = 0 .
\label{M=2:7thIntegral}
\end{multline}
The sum of these later three is zero modulo \eqref{M=2:4thIntegral} and \eqref{M=2:1stIntegral}.
The last integral of the motion is 
\begin{multline}
   e_3+\xi_0-\eta_2-\eta_0\xi_1-\xi_2\eta_1-x_2y_0+\eta_0x_2y_1-\xi_2x_1y_0+\eta_0\xi_2x_1y_1
\\
   -\xi_1x_0y_0+(\xi_0-\eta_2-\xi_2\eta_1)x_0y_1+\xi_1\eta_1x_0y_2
   +(\xi_0-\eta_2-\eta_0\xi_1)x_1y_2+\eta_1x_2y_2 = 0 . 
\label{M=2:8thIntegral}
\end{multline}
Note that we have revealed the
appearance of all of the three elementary symmetric functions of independent parameters $ \nu_0, \nu_1, \nu_2 $. 
\end{proposition}
\begin{proof}
Comparing \eqref{M=2ODE:9} and \eqref{M=2ODE:10} and noting the initial values for $ \xi_2 $ and $ \eta_0 $ as given by
\eqref{M=2:eta0} and \eqref{M=2:xi2} along with the assumption $ \min({\rm Re}(\nu_1-\nu_0),{\rm Re}(\nu_2-\nu_0)) > -1 $
we have \eqref{M=2:1stIntegral}. Considering \eqref{M=2:4thIntegral} next we compute $ s $ times the derivative of $ x_0y_2 $
using \eqref{M=2ODE:1} and \eqref{M=2ODE:4} and find
\begin{align*}
   (sx_0y_2)^{\prime} & = s(x_0y_2)^{\prime}+x_0y_2 
   \\
   & = -\eta_0x_0y_2-\xi_2x_0y_2-x_1y_2+x_0y_1+x_0y_2
   \\
   & = \eta_0\xi_2^{\prime}+\eta_0^{\prime}\xi_2+\eta_1^{\prime}-\xi_1^{\prime}-\eta_0^{\prime}
   \\
   & = (\eta_0\xi_2+\eta_1-\xi_1-\eta_0)^{\prime} .
\end{align*}
Assuming again $ \min({\rm Re}(\nu_1-\nu_0),{\rm Re}(\nu_2-\nu_0)) > -1 $ we can fix the integration constant and deduce
\eqref{M=2:4thIntegral}.

Computing $ s $ times the derivative of $ x_{0}y_{0}+x_{1}y_{1}+x_{2}y_{2} $ using \eqref{M=2ODE:1}-\eqref{M=2ODE:6} we find
this vanishes and if $ s \neq 0 $ then this quantity is a constant. Assuming $ \nu_0 \neq 0 $, $ \min({\rm Re}(\nu_1-\nu_0),{\rm Re}(\nu_2-\nu_0)) > 0 $,
or if $ \nu_0 = 0 $ then this lower bound can be dropped to  $ -1 $, then the inner product vanishes as $ s \to 0 $ and thus
the constant is in fact zero. Alternatively one can deduce $ {\rm Tr}A^{(2)} = 0 $ from \eqref{schlesinger:1}.

Next we derive \eqref{M=2:5thIntegral}. We first rewrite \eqref{M=2:3rdIntegral} in the following way
\begin{align}
 0 & = x_0y_0+x_2y_2+\frac{(x_1y_2)(x_0y_1)}{x_0y_2} ,
\nonumber   \\
   & = -\xi_0^{\prime}-\eta_2^{\prime}-\frac{\eta_1^{\prime}\xi_1^{\prime}}{\eta_0^{\prime}} .
\label{tmp:3}
\end{align}
Now we seek alternative forms for $ \eta_1^{\prime}\xi_1^{\prime} $ and to this end we re-examine \eqref{M=2:3rdIntegral} from
a different point of view. Using the formulae for $ x_1, x_2, y_0, y_1 $ given in \eqref{aux:x1},\eqref{aux:x2},\eqref{aux:y0},\eqref{aux:y1} we
rewrite the orthogonality relation as
\begin{equation}
   0 = x_0 y_2 \left(e_2-\eta_0-3 s x_0 y_2\right)+\left(e_1+1\right) s y_2 x_0'+\left(1-e_1\right) s x_0 y_2'-s^2x_0'y_2'+s^2 y_2 x_0''+s^2 x_0 y_2'' ,
\label{tmp:1}
\end{equation}
and using $ \eta_0^{(3)}+2 x_0' y_2'+y_2 x_0''+x_0 y_2'' = 0 $ we can eliminate the last two terms of the above in favour
of $ x_0' y_2' $ which gives
\begin{equation}
   0 = x_0 y_2 \left(e_2-\eta_0-3 s x_0 y_2\right)-s^2\eta_0^{(3)}+\left(e_1+1\right) s y_2 x_0'+\left(1-e_1\right) s x_0 y_2'-3s^2x_0'y_2' .
\label{tmp:2}
\end{equation}
Now using the identity
\begin{equation}
    s^2x_0' y_2' = \frac{(s y_2 x_0')(s x_0 y_2')}{x_0y_2} ,
\label{tmp:4}
\end{equation}
in \eqref{tmp:2} and solutions of \eqref{aux:xi1} and \eqref{aux:eta1} for $ s x_0 y_2' $ and $ s y_2 x_0' $ respectively
we arrive at an alternative form for the orthogonality relation
\begin{multline*}
   0 = -\eta_0' \left(-e_1^2 \eta_0'+e_2 \eta_0'+e_1 (3 \eta_0+1) \eta_0'+s^2 \eta_0^{(3)}+3 s \eta_0'^2-3 \eta_0^2 \eta_0'-3 \eta_0 \eta_0'\right)
   \\
   +\left(e_1-3 \eta_0-1\right) \eta_0' \xi_1'+\left(-2 e_1+3 \eta_0+1\right) \eta_0' \eta_1'-3 \eta_1' \xi_1' .
\end{multline*}
We solve this for $ \eta_1' \xi_1' $ and substitute this into \eqref{tmp:3} yielding
\begin{equation*}
  0 = e_1 \left(\eta_0'+2 \eta_1'-\xi_1'\right)-e_1^2 \eta_0'+e_2 \eta_0'+s^2 \eta_0^{(3)}+3 s \eta_0 \eta_0''+3 s \eta_0'^2+3 \eta_0^2 \eta_0'-3 \eta_0 \eta_0'-\eta_1'+\xi_1'-3 \eta_2'-3 \xi_0' .
\end{equation*} 
This is a perfect derivative and when integrated after noting the $ s \to 0 $ limits of \eqref{M=2:xi0} and \eqref{M=2:xi1}
(with the proviso $ \min({\rm Re}(\nu_1-\nu_0),{\rm Re}(\nu_2-\nu_0)) > -1 $), we obtain
\begin{multline}
  0 = -3e_3-(1-e_1)e_2+(-e_1^2+e_1+e_2+2)\eta_0-3\eta_0^2+\eta_0^3+(3\eta_0-2)s\eta_0'+s^2\eta_0''
  \\
  +(2e_1-1)\eta_1+(1-e_1)\xi_1-3\eta_2-3\xi_0 .
\label{M=2orthog}
\end{multline} 
Now \eqref{M=2:5thIntegral} immediately follows by substituting for $ \eta_0' $ and $ \eta_0'' $ using \eqref{M=2ODE:10},
\eqref{M=2ODE:1} and \eqref{M=2ODE:4} to clear all the derivatives.
However, as alluded to in the proposition, we can go further and split this relation. 

We intend to integrate \eqref{aux:xi0} in order to prove \eqref{M=2:6thIntegral}. 
The first thing we do is use the identity for $ s^2x_0y_2'' $
\begin{equation*}
   s^2x_0y_2'' = s\left( sx_0y_2^{\prime} \right)^{\prime}-sx_0y_2^{\prime}-s^2x_0^{\prime}y_2^{\prime} ,
\end{equation*}
to replace $ s^2x_0y_2'' $ in \eqref{aux:xi0}. Next we replace the term $ s^2x_0^{\prime}y_2^{\prime} $ using \eqref{tmp:2}.
This leaves us with terms linear in $ sx_0^{\prime}y_2 $ and $ sx_0y_2^{\prime} $ and we replace these last two factors
by solving \eqref{aux:eta1} and \eqref{aux:xi1} respectively. The end result is
\begin{multline*}
  0 = 2 e_1^2 \eta_0'+e_1 \eta_0'+e_2 \eta_0'-6 e_1 \eta_0 \eta_0'+3 \eta_0^2 \eta_0'-3 \eta_0 \eta_0'+s^2 \eta_0^{(3)}
  \\
  -e_1 \eta_1'-\eta_1'+2 e_1 \xi_1'-3 \xi_1 \eta_0'-3 \eta_0 \xi_1'+\xi_1'+3 s(sx_0y_2^{\prime})^{\prime}+3 \xi_0' .
\end{multline*}
This is a perfect derivative whose integral is determined as
\begin{multline}
  0 = -2 \left(e_1+2\right) e_2+3 e_3+\left(2 e_1^2+4 e_1+e_2+2\right) \eta_0-3 \left(e_1+1\right) \eta_0^2+\eta_0^3-2 s\eta_0'+ s^2\eta_0''
  \\
  +\left(-e_1-1\right) \eta_1+\left(2 e_1-3 \eta_0+4\right)\xi_1 +3 s^2x_0y_2^{\prime}+3 \xi_0 .
\label{alt-6th}
\end{multline}
Here the initial conditions \eqref{M=2:xi0} and \eqref{M=2:xi1} have been employed under the assumption
$ \min({\rm Re}(\nu_1-\nu_0),{\rm Re}(\nu_2-\nu_0)) > -1 $ and $ \nu_0 \neq 0 $.
Clearing the derivatives of $ \eta_0 $ and subsequent derivatives from this expression gives \eqref{M=2:6thIntegral}.
The method for proving \eqref{M=2:7thIntegral} is similar and will entail integrating \eqref{aux:eta2}. 
Here we use the identity for $ s^2x_0''y_2 $
\begin{equation*}
   s^2x_0''y_2 = s\left( sx_0^{\prime}y_2 \right)^{\prime}-sx_0^{\prime}y_2-s^2x_0^{\prime}y_2^{\prime} ,
\end{equation*}
to replace $ s^2x_0''y_2 $ in \eqref{aux:eta2}. Again we replace the term $ s^2x_0^{\prime}y_2^{\prime} $ using \eqref{tmp:2}.
This also leaves us with terms linear in $ sx_0^{\prime}y_2 $ and $ sx_0y_2^{\prime} $ and we replace these last two factors
by solving \eqref{aux:eta1} and \eqref{aux:xi1} respectively. Our result this time is
\begin{multline*}
   0 = -e_1^2 \eta_0'+e_1 \eta_0'+e_2 \eta_0'+3 \eta_0^2 \eta_0'-3 \eta_0 \eta_0'+s^2 \eta_0^{(3)}
   \\
   +3 \eta_0 \eta_1'+3 \eta_1 \eta_0'-e_1 \eta_1'-\eta_1'-e_1 \xi_1'+\xi_1'+3 s (sx_0^{\prime}y_2)'+3 \eta_2' .
\end{multline*}
This is a perfect derivative whose integral is determined as
\begin{multline}
  0 = -\left(1-e_1\right) e_2+\left(e_1+e_2+2-e_1^2\right) \eta_0-3 \eta_0^2+\eta_0^3-2 s \eta_0'+s^2 \eta_0''
  \\
  +\eta_1 \left(-e_1+3 \eta_0-4\right)+\left(1-e_1\right) \xi_1+3 s^2x_0^{\prime}y_2+3 \eta_2 .
\label{alt-7th}
\end{multline}
Here the initial condition \eqref{M=2:xi1} has been employed under the previous assumptions.
Clearing the derivatives of $ \eta_0 $ and subsequent derivatives from this expression gives \eqref{M=2:7thIntegral}.
The last integral of the motion, \eqref{M=2:8thIntegral}, is $ \det( C-A^{(2)} ) + e_3 $. One can verify
directly it is a constant using the equations of motion \eqref{M=2ODE:1}-\eqref{M=2ODE:12}.
\end{proof} 

\begin{proposition}
Alternatives to the identity \eqref{M=2:4thIntegral} are the relations
\begin{gather}
   e_3 - sx_{1}y_{2} + 2\eta_{2} + \xi_{0} - \eta_{1} + \eta_{1}\xi_{2} = (\eta_{0}-2) \left[ -e_2 + sx_{0}y_{2} + \eta_{0} - \eta_{1} + \xi_{1} - \eta_{0}\xi_{2} \right] = 0 ,
\label{Idsx1y2}\\
   e_2 - 2e_3 - sx_{0}y_{1} - \eta_{2} - 2\xi_{0} - \xi_{1} + \eta_{0}\xi_{1} = (2+e_1-\eta_{0}) \left[ -e_2 + sx_{0}y_{2} + \eta_{0} - \eta_{1} + \xi_{1} - \eta_{0}\xi_{2} \right] = 0 .
\label{Idsx0y1}
\end{gather}
\end{proposition}
\begin{proof}
The proof employed for $ sx_0y_2 $ can be easily adapted to $ sx_0y_1 $ and $ sx_1y_2 $. We observe
\begin{align*}
   (sx_0y_1)^{\prime} & = s(x_0y_1)^{\prime}+x_0y_1 
   \\
   & = -\eta_0x_0y_1-x_1y_1-\xi_1x_0y_2+x_0y_0+x_0y_1
   \\
   & = \eta_0\xi_1^{\prime}-\xi_0^{\prime}-\eta_2^{\prime}+\xi_1\eta_0^{\prime}-\xi_1^{\prime}-\xi_0^{\prime}
   \\
   & = (\eta_0\xi_1-\xi_1-2\xi_0-\eta_2)^{\prime} ,
\end{align*}
and
\begin{align*}
   (sx_1y_2)^{\prime} & = s(x_1y_2)^{\prime}+x_1y_2 
   \\
   & = -\eta_1x_0y_2-x_2y_2-\xi_2x_1y_2+x_1y_1+x_1y_2
   \\
   & = \eta_1\xi_2^{\prime}+\eta_2^{\prime}+\xi_2\eta_1^{\prime}+\xi_0^{\prime}+\eta_2^{\prime}-\eta_1^{\prime}
   \\
   & = (\eta_1\xi_2+2\eta_2+\xi_0-\eta_1)^{\prime} .
\end{align*}
These two relations are not independent of \eqref{M=2:4thIntegral} as can be seen by the following argument.
For $ sx_0y_2 $ we have
\begin{equation*}
 \left(e_1+1\right) \eta_0-e_2-s \eta_0'-\eta_0^2-\eta _1+\xi _1 = 0 ,
\end{equation*}
whilst for $ sx_1y_2 $
\begin{equation*}
 -\eta_0 \left(2 e_1+s \eta_0'+2\right)+\left(e_1+3\right) \eta_0^2-e_2 \left(\eta_0-2\right)+2 s \eta_0'-\eta_0^3
    +\left(\eta_0-2\right) \xi _1+\left(2-\eta_0\right) \eta _1 = 0 ,
\end{equation*}
and for $ sx_0y_1 $
\begin{multline*}
  \eta_0 \left(e_1^2+3 e_1+s \eta_0'+2\right)-\left(e_1+2\right) s \eta_0'-\left(2 e_1+3\right) \eta_0^2+e_2 \left(-e_1+\eta_0-2\right)+\eta_0^3
\\
  +\xi _1 \left(e_1-\eta_0+2\right) +\eta _1 \left(-e_1+\eta_0-2\right) = 0 .
\end{multline*}
The factorisation of these two relations gives \eqref{Idsx1y2} and \eqref{Idsx0y1}.
\end{proof}

\begin{proposition}\label{eta_form}
Define the radical $ F $ by
\begin{equation}
  F^2 \coloneqq 4 e_1^2 \eta_0'^2 - 12 e_2 \eta_0'^2 + 12 \eta_0 \eta_0'^2 - 36s \eta_0'^3 + 9 s^2 \eta_0''^2 - 12s \eta_0' (\eta_0''+s \eta_0^{(3)}) .
\label{Fsquared}
\end{equation} 
The resolvent function $ \eta_0(s) $ satisfies a scalar ordinary differential equation with degrees $ 2,3,4,8 $ in 
$ \eta_0^{(4)}, \eta_0^{(3)}, \eta_0^{(2)}, \eta_0^{(1)} $ 
\begin{multline}
   27 s^6 \left(\eta_0^{(4)}\right)^2 \eta_0'^2
\\
   +27 s^4 \left[ -F \eta_0''+3 s^2 \eta_0''^3+6 s \eta_0'^3 \eta_0''+2 \eta_0'^2 \left(\eta_0''+3 s \eta_0^{(3)}\right)
          \right. \\ \left.
          -5 s \eta_0' \eta_0'' \left(\eta_0''+s \eta_0^{(3)}\right)+4 \eta_0'^4 \right] \eta_0^{(4)} 
\\
   +81 s^6 \left(\eta_0^{(3)}\right)^3 \eta_0' 
   \\
   +\left[ -27 e_1^2 s^4 \eta_0'^2+81 e_2 s^4 \eta_0'^2+18 F s^4-54 s^6 \eta_0''^2-162 s^5 \eta_0' \eta_0''
   \right. \\ \left.
   +567 s^5 \eta_0'^3-81 s^4 \eta_0 \eta_0'^2+243 s^4 \eta_0'^2 \right]\left(\eta_0^{(3)}\right)^2 
\\
   -3 s^2 \left[ F \left(15 s \eta_0''-2 \eta_0' \left(e_1^2-3 \left(e_2+3 s \eta_0'-7 \eta_0\right)\right)\right)
   \right. \\ \left.
   +9 s^2 \eta_0' \eta_0''^2 \left(-2 e_1^2+6 e_2+54 s \eta_0'-6 \eta_0+11\right)+4 \eta_0' \left(9 s \left(e_1^2-3 e_2+3 \eta_0-3\right) \eta_0'^3
   \right.\right. \\ \left.
   -\left(27 \left(e_3+s\right)+2 e_1^3-9 e_2 e_1\right) \eta_0'-108 s^2 \eta_0'^4+27 \eta_0\right)
   \\ \left.
   -18 s \eta_0'^2 \eta_0'' \left(-e_1^2+3 e_2+25 s \eta_0'-3 \eta_0+3\right)-45 s^3 \eta_0''^3 \right]\eta_0^{(3)} 
\\
   +27 s^4 \left(-e_1^2+3 e_2+27 s \eta_0'-3 \eta_0+1\right)\eta_0''^4 
\\
   -54 s^3 \eta_0' \left(-e_1^2+3 e_2+24 s \eta_0'-3 \eta_0+1\right) \eta_0''^3
\\
   -9 s^2 \left[ F \left(-e_1^2+3 e_2+18 s \eta_0'+6 \eta_0+1\right)-3 s \left(4 e_1^2-12 e_2+12 \eta_0+17\right) \eta_0'^3
   \right. \\ \left.
   +3 \left(e_1^2-3 e_2+3 \eta_0-1\right) \eta_0'^2+\left(27 \left(e_3+s\right)+2 e_1^3-9 e_2 e_1\right) \eta_0'+108 s^2 \eta_0'^4-27 \eta_0 \right] \eta_0''^2
\\
   +6 s \eta_0' \left[ F \left(e_1^2-3 \left(e_2+6 s \eta_0'-7 \eta_0\right)\right)-18 s \left(e_1^2-3 e_2+3 \eta_0-1\right) \eta_0'^3
   \right. \\ \left.
   +2 \left(27 \left(e_3+s\right)+2 e_1^3-9 e_2 e_1\right) \eta_0'+270 s^2 \eta_0'^4-54 \eta_0 \right]\eta_0'' 
\\
   -4 \eta_0'^2 \left[ F \left(e_1^2-3 e_2-9 s \eta_0'+3 \eta_0\right) \left(e_1^2-3 \left(e_2+3 s \eta_0'-4 \eta_0\right)\right)
   \right. \\
   +27 s^2 \left(e_1^2-3 e_2+3 \eta_0-1\right) \eta_0'^4-9 s \left(27 \left(e_3+s\right)+2 e_1^3-9 e_2 e_1\right) \eta_0'^2
   \\
   +\left(3 \left(27 \left(e_3+4 s\right)+2 e_1^3-9 e_2 e_1\right) \eta_0+\left(e_1^2-3 e_2\right) \left(27 \left(e_3+s\right)+2 e_1^3-9 e_2 e_1\right)\right) \eta_0'
   \\ \left.
   -27 \eta_0 \left(e_1^2-3 e_2+3 \eta_0\right)-243 s^3 \eta_0'^5 \right]
    = 0 .  
\label{eta0_ODE}
\end{multline}
Here $ F $ is defined as the positive root of \eqref{Fsquared}.
\end{proposition}
\begin{proof}
For the sake of notational simplicity we define the abbreviations
\begin{equation*}
   U \coloneqq s x_0 y_2',\quad V \coloneqq s x_0' y_2,\quad W \coloneqq s^2 x_0 y_2'',\quad Z \coloneqq s^2 x_0'' y_2 .
\end{equation*}
Our derivation entails two steps. The first step is to express the auxiliary variables $ \xi_1, \eta_1, \xi_0, \eta_2 $
in terms of $ U,V,W,Z $ and for this we need four relations to solve. We take these four to be the relations
\eqref{M=2:4thIntegral}, \eqref{M=2:6thIntegral}, \eqref{M=2:7thIntegral}  and \eqref{M=2:8thIntegral}. In
each of these we replace the bi-linear products $ x_jy_k $ using, for example 
\begin{align}
  x_0y_1 & = -(\eta_0-e_1)\eta_0' + U,
\label{aux:x0y1}\\
  x_1y_2 & = \eta_0\eta_0' - V,
\label{aux:x1y2}\\
  x_0y_0 & = -s\eta_0'{}^2-\eta_0' \xi_1 + (1+\eta_0-e_1)U + W,
\label{aux:x0y0}\\
  x_2y_2 & = \eta_0'\eta_1-s\eta_0'{}^2 + (1+\eta_0)V + Z,
\label{aux:x2y2}
\end{align}
which follow by writing $ x_1,x_2,y_1,y_0 $ using \eqref{aux:x1},\eqref{aux:x2}, \eqref{aux:y1} and \eqref{aux:y0}
and then rewriting the derivatives of $ x_0, y_2 $ using the abbreviations. In this way we have four independent
inhomogeneous relations which are linear in $ \xi_1, \eta_1, \xi_0, \eta_2 $, and have a unique solution
assuming $ \eta _0'(1+e_1 \eta _0'+U-V) \neq 0 $.

In the second step our strategy is to seek an elimination scheme for $ U, V, W, Z $ and for this we require four
equations involving these variables. Firstly we differentiate \eqref{M=2ODE:10} and this gives us
\begin{equation}
  U+V+s \eta_0'' = 0 .
\label{UVsum}
\end{equation}
Next we employ \eqref{tmp:4} and \eqref{M=2ODE:10} in \eqref{tmp:2} and deduce
\begin{equation}
  \frac{3 U V}{\eta_0'}+\left(1-e_1\right) U+\left(e_1+1\right) V-\eta_0' \left(e_2-\eta_0+3 s \eta_0'\right)-s^2\eta_0^{(3)} = 0 .
\label{UVprod}
\end{equation}
These two relations allow us to solve for $ U, V $ leading to a quadratic equation and the appearance of the
radical $ F $. If we differentiate \eqref{M=2ODE:10} once more and utilise \eqref{tmp:4} again we find
\begin{equation}
   W+Z-\frac{2 U V}{\eta_0'}+s^2 \eta_0^{(3)} = 0 .
\label{WZsum}
\end{equation}
We construct our last relation by adding $ x_0 $ times \eqref{M=2:third:2} to $ y_2 $ times \eqref{M=2:third:1}
and utilise the third derivative of \eqref{M=2ODE:10} to eliminate $ x_0'''y_2+x_0y_2''' $ from this result. We 
then use the identity
\begin{equation*}
    x_0'' y_2' = \frac{(s^2 x_0''y_2)(s x_0y_2')}{s^3 x_0y_2} , 
\end{equation*}
and a similar one for $ x_0' y_2'' $ to conclude
\begin{multline}
\frac{3 (U Z+V W)}{\eta_0'}+3 (W+Z)-e_1 (W-Z)+ \left(1+e_2-\eta_0+6 s \eta_0'\right)(U+V)-e_1 (U-V)
\\
-2 s \eta_0'^2-s^3 \eta_0^{(4)} = 0 .
\label{WZcombo}
\end{multline}
These four relations allow us to eliminate all reference to $ x_0, y_2 $ and their derivatives in favour of 
$ \eta_0 $ and its first four derivatives via the quantities $ U,V,W,Z $.
Substituting the solution for $ \xi_1, \eta_1, \xi_0, \eta_2 $ found in the first step, and then the solution 
for $ U,V,W,Z $ in the second step, into the energy conservation relation \eqref{M=2:Energy} now expressed as
\begin{equation}
  U Z-V W +e_1 U V-\eta_0'(s)\left[ \eta_0+s(U-V) \eta_0'(s) \right] + \eta_0'{}^2 \left[ -e_1 \eta_1+\eta_0(\eta_1+\xi_1)+\eta_2-\xi_0+s \right]= 0 ,
\label{alt_Ham}
\end{equation} 
we find that the final result is \eqref{eta0_ODE}.
\end{proof}

\begin{lemma}
The quantity $ F^2 $ is a perfect square and the radical $ F $  is 
\begin{equation}
F = - 3 x_0y_1 - 3 x_1y_2 - e_1 x_0y_2 .
\end{equation}
The sign is chosen here so that $ F>0 $ for the appropriate solution to the boundary conditions \eqref{x0_IC}-\eqref{y2_IC}.
\end{lemma}
\begin{proof}
We will prove this by way of verification. Let us use the abbreviations $ U, V $ as in the previous proof. Now
\begin{equation*}
   F = -3x_0y_1 - 3x_1y_2 - e_1x_0y_2 
     = -3x_0(\xi_2y_2+sy_2')-3(-\eta_0x_0-sx_0')y_2-e_1x_0y_2 .
\end{equation*}
This can be further rewritten 
\begin{equation}
   F = 3x_0y_2(-\xi_2+\eta_0)-e_1x_0y_2-3sx_0y_2'+3sx_0'y_2
     = -2e_1\eta_0'+3(V-U) .
\label{FUV}
\end{equation}
Employing the identity $ (V-U)^2 = (V+U)^2 - 4UV $ and the above relation for the difference, \eqref{UVsum} for the sum
and \eqref{UVprod} for the product we readily compute that $ F $ satisfies the definition \eqref{Fsquared}.
\end{proof}

In the Okamoto theory of the Painlev\'e equations expressing the Hamiltonian co-ordinates and momenta in terms of
the Hamiltonian and its derivatives is an important task. For PIII', or the $ M=1 $ systems, this was given in Remark \ref{M=1Ham},
and the analogous result for the $ M=2 $ system is given in the Appendix.
 
\subsection{Behaviour of \texorpdfstring{$ \eta_0 $}{ZeroExpansions} at \texorpdfstring{$ s\to 0 $}{ZeroExpansions} and \texorpdfstring{$ s\to \infty $}{InftyExpansions}}
 
Having derived the scalar ordinary differential equation \eqref{eta0_ODE}, equivalent to the coupled first order system
\eqref{M=2ODE:1} to \eqref{M=2ODE:12}, we can employ this form to good advantage in the analysis of the solutions in the 
neighbourhood of the singular points $ s = 0 $ and $ s = \infty $.
We consider the singular point at $ s = 0 $ first, which in our theoretical construction occupies the special
place by defining the precise solutions we seek as one can observe from \eqref{M=2:eta0}. 
However we will undertake the task of this analysis in the generic situation and therefore
encounter other classes which are not directly relevant to the original problem at hand.

\begin{proposition}
About the singular point $ s=0 $ \eqref{eta0_ODE} possesses various solution types of the form 
\begin{equation}
\eta_0 = C_1 s^{\lambda_1} + C_2 s^{\lambda_1+\delta_1} + {\rm O}(s^{\lambda_1+\delta_1+\epsilon}) ,
\label{eta0_expansion}
\end{equation}
where $ {\rm Re}\, \lambda_1 > 0 $, $ {\rm Re}\, \delta_1 > 0 $ and $ {\rm Re}\, \epsilon > 0 $.
The first class have exponents fixed by the parameters in the leading order with
\begin{equation}
	\lambda_1 = \begin{cases}
	\nu_0-\nu_1+1 & \\
	-\nu_0+\nu_1+1 & \\
	\nu_0-\nu_2+1 & \\
	-\nu_0+\nu_2+1 & \\
	\nu_1-\nu_2+1 & \\
	-\nu_1+\nu_2+1 & 
	\end{cases} ,
	\label{s=0:b}
\end{equation} 
where $ C_1 \neq 0 $ is arbitrary and include the case at hand of \eqref{M=2:eta0}. 
In addition there is the case $ \lambda_1 = 0 $ and two further cases with exponents determined by the parameters
\begin{equation}
	\lambda_1 = 1\pm\frac{2 \sqrt{\nu_0^2+\nu_1^2+\nu_2^2-\left(\nu_1+\nu_2\right) \nu_0-\nu_1 \nu_2}}{\sqrt{3}} ,
	\label{s=0:c}
\end{equation} 
and
\begin{equation}
	\lambda_1 = \frac{1}{6} \left(3\pm\sqrt{3} \sqrt{4 \nu_0^2+4 \nu_1^2+4 \nu_2^2-4 \left(\nu_1+\nu_2\right) \nu_0-4 \nu_1 \nu_2-1}\right) ,
	\label{s=0:d} 
\end{equation}
where again $ C_1 \neq 0 $ is arbitrary.
The last class have rational, i.e. fractional exponents, at the leading order
\begin{equation}
    \eta_0 \sim \sqrt[3]{\pm\sqrt{x^2+y}-x}\;s^{1/3} ,
\label{s=0:a}
\end{equation}
where $ C_1 $ is fixed by the parameters. Here 
\begin{equation}
    x = \left(\nu_0+\nu_1-2 \nu_2\right) \left(2 \nu_0-\nu_1-\nu_2\right) \left(\nu_0-2 \nu_1+\nu_2\right) ,
\label{xDefn}
\end{equation} 
and 
\begin{equation}
    y = \frac{1}{27} \left(9 \left(\nu_0-\nu_1\right)^2-4\right) \left(9 \left(\nu_0-\nu_2\right)^2-4\right) \left(9 \left(\nu_1-\nu_2\right)^2-4\right) .
\label{yDefn}
\end{equation}
\end{proposition}
\begin{proof}
First let us render the non-linear ODE \eqref{eta0_ODE} in a form which is a polynomial in all derivatives of $ \eta_0 $.
This entails solving \eqref{eta0_ODE} for the radical F, squaring the result and equating this to the right-hand side of
\eqref{Fsquared}. We do not display this because of its size and refer to it as $ P\star $.
We employ the algebraic expansion \eqref{eta0_expansion} and examine a region of the convex hull of the points in 
Fig. \ref{lead_exponents} on the lower left-hand boundary.
 
\begin{figure}[!h]
\includegraphics[width=0.5\textwidth]{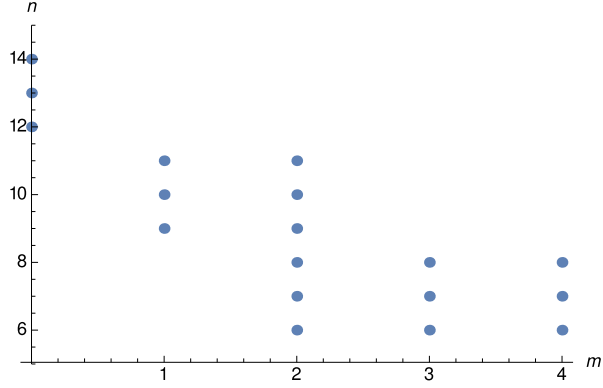}
\caption{Newton polygon of the exponents $ (m,n) $ for the leading term $ s^{m+n\lambda_1} $ of an algebraic 
expansion given by \eqref{eta0_expansion}.}
\label{lead_exponents}
\end{figure}

If one takes the lower corner point $ 6 \lambda_1+2 $ alone then there are 126 terms contributing, which sum to
\begin{multline*}
 (\lambda_1+\nu_0-\nu_1-1) (\lambda_1-\nu_0+\nu_1-1) (\lambda_1+\nu_0-\nu_2-1) (\lambda_1+\nu_1-\nu_2-1) (\lambda_1-\nu_0+\nu_2-1) (\lambda_1-\nu_1+\nu_2-1) 
 \\ \times
 27 C_1^6 \lambda_1^6 \left(3 \lambda_1^2-6 \lambda_1-4 \nu_0^2-4 \nu_1^2-4 \nu_2^2+4 \nu_0 \nu_1+4 \nu_0 \nu_2+4 \nu_1 \nu_2+3\right)^2 s^{6 \lambda_1+2} .
\end{multline*}
These require $ C_1 \neq 0 $ but otherwise arbitrary and are given in \eqref{s=0:b} and \eqref{s=0:c}.
Another solution derives from the single point condition at $ 12 \lambda_1 $ and the 14 terms give
\begin{equation*}
 11664 C_1^{12} \lambda_1^{12} \left(3 \lambda_1^2-3 \lambda_1-\nu_0^2-\nu_1^2-\nu_2^2+\nu_0 \nu_1+\nu_0 \nu_2+\nu_1 \nu_2+1\right)^2 s^{12 \lambda_1} .
\end{equation*}
These solutions are given in \eqref{s=0:d}. In addition if the condition at $ 9 \lambda_1+1 $ applies then we have 37 terms
contributing to yield
\begin{multline*}
 216 C_1^9 \lambda_1^9 \left(\nu_0+\nu_1-2 \nu_2\right) \left(2 \nu_0-\nu_1-\nu_2\right) \left(\nu_0-2 \nu_1+\nu_2\right)
 \\ \times
 \left(3 \lambda_1^2-6 \lambda_1-4 \nu_0^2-4 \nu_1^2-4 \nu_2^2+4 \nu_0 \nu_1+4 \nu_0 \nu_2+4 \nu_1 \nu_2+3\right)
 \\ \times
 \left(3 \lambda_1^2-3 \lambda_1-\nu_0^2-\nu_1^2-\nu_2^2+\nu_0 \nu_1+\nu_0 \nu_2+\nu_1 \nu_2+1\right) s^{9 \lambda_1+1} ,
\end{multline*}
and the solutions \eqref{s=0:c} and \eqref{s=0:d} appear again.

However in addition to these there is another class of non-analytic solutions.
If we demand the equality of the three points $ 12 \lambda_1=9 \lambda_1+1=6 \lambda_1+2 $ then we
deduce $ \lambda_1 = \frac{1}{3} $. There are 177 terms contributing at these three points and their sum is
\begin{multline*}
	\frac{16}{177147} \left(3 \nu_0^2-3 \nu_1 \nu_0-3 \nu_2 \nu_0+3 \nu_1^2+3 \nu_2^2-3 \nu_1 \nu_2-1\right)^2 s^4C_1^6 
	\\ \times
	\left[
	27 C_1^6+54 C_1^3 \left(2 \nu_0-\nu_1-\nu_2\right) \left(\nu_0-2 \nu_1+\nu_2\right) \left(\nu_0+\nu_1-2 \nu_2\right)
	\right.
	\\ \left.   
	-\left(9 \left(\nu_0-\nu_1\right)^2-4\right) \left(9 \left(\nu_0-\nu_2\right)^2-4\right) \left(9 \left(\nu_1-\nu_2\right)^2-4\right)
	\right] ,
\end{multline*}
and the non-trivial solution for $ C_1 $ is given by the equation $ 27 C_1^6+54 C_1^3 x-27 y $. These are the 
{\it fractional exponent solutions} in \eqref{s=0:a}.
\end{proof}

Next we consider $ s = \infty $ and examine the generic asymptotic solution developed about this point.
\begin{proposition}\label{sLarge}
As $ s \to \infty $ and $ {\rm arg}(s) < \frac{3}{4}\pi $ the solution of \eqref{eta0_ODE} for a general resolvent function 
$ \eta_0 $ permits the asymptotic expansion 
\begin{equation}
   \eta_0(s) = -\frac{3}{2^{4/3}}s^{2/3} + {\rm O}(s^{1/3},1/\log(s)) .
\label{eta0_Large_s}
\end{equation} 
\end{proposition}
\begin{proof}
Let us determine the necessary conditions for a large-$s$ algebraic solution of the form \eqref{eta0_expansion} 
to equation $ P\star $. Employing just the first 
term in $ P\star $ we find that $ s^8 $ times this expression possesses 1923 terms with an $s$-dependence of the form
$ s^{m+n\lambda_1} $ with $ m \in \Z_{\geq 0} $ and $ n \in \N $. A consolidated plot of these $ (m,n) $ values is given
in Fig. \ref{lead_exponents}.

From this plot it is clear the line defined by the points $ 14\lambda_1, 2+11\lambda_1, 4+8\lambda_1 $ defines an upper 
right-hand segment of the convex hull of all these points. Of necessity it must have negative slope. 
The single mutual solution for $ \lambda_1 $ by equating each of these
with the others is $ \lambda_1=2/3 $. There are 14 terms associated with these three points and their sum gives, after
making the substitution for the $ \lambda_1 $ solution,
\begin{equation*}
  \frac{256}{81}s^{28/3} C_1^8 \left(16 C_1^3+27\right)^2 .
\end{equation*} 
The only acceptable, real and non-zero solution for the coefficient is $ C_1 = -3 \times 2^{-4/3} $. 
Proceeding on we introduce an algebraic sub-leading term, as in \eqref{eta0_expansion}, and specialise the 
values for the exponent and coefficient of leading term found earlier. When we examine $ s^{4/3} $ times this resulting
expression we find 28042 terms of the form $ s^{m+n\delta_1} $ with $ m \in \frac{1}{3}\Z $ and $ n \in \Z_{\geq 0} $.
The consolidated plot of these $ (m,n) $ values is given in Fig. \ref{sub-lead_exponents}.

\begin{figure}[H]
\includegraphics[width=0.75\textwidth]{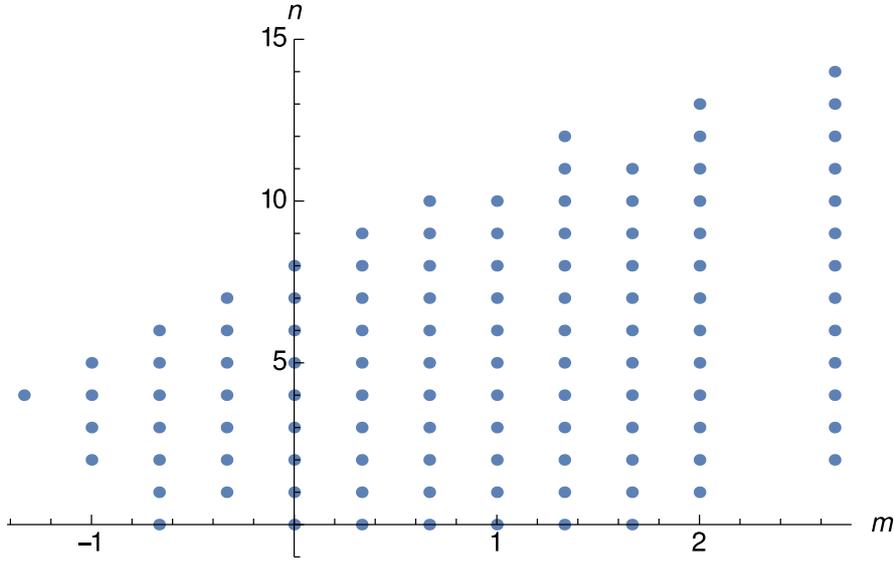}
\caption{Newton polygon of the exponents $ s^{m+n\delta_1} $ for the sub-leading term of asymptotic expansion of \eqref{eta0_expansion}.}
\label{sub-lead_exponents}
\end{figure}

Considering this figure we observe that there are two possibilities for lines defining an upper boundary to the convex
hull of these points, both with positive slope. The first of the two is defined by the points 
$ \frac{8}{3}+14\delta_1, 2+13\delta_1, \frac{4}{3}+12\delta_1 $ and yields the solution $ \delta_1 = -2/3 $. 
However the total of the 378 terms which contribute to this vanish identically with this solution for the exponent and
so the coefficient is undetermined. The second of the two lines is defined by the set of 8 points
$ \frac{4}{3}+12\delta_1, \frac{2}{3}+10\delta_1, \frac{1}{3}+9\delta_1, 8\delta_1, -\frac{1}{3}+7\delta_1, -\frac{2}{3}+6\delta_1, -1+5\delta_1, -\frac{4}{3}+4\delta_1 $
and their mutual equality gives the solution $ \delta_1 = -1/3 $. There are 3915 terms which have these exponents and
their sum, under evaluation of $ \delta_1 $, is non-zero.

If we admit an algebraic-logarithmic sub-leading term of the form
\begin{equation}
  \eta_0 = C_1 s^{\lambda_1} + C_2 s^{\lambda_1+\delta_1}\left( \log s \right)^{\mu_1} ,
\label{eta0_alg+log}
\end{equation} 
$  $ 
and employ the solution for the leading term then we have 384003 terms of the form $ s^{m+n\delta_1}t^{k+l\mu_1} $ where 
$ t \coloneqq \log s $ and $ \mu_1 \neq 0 $.

\begin{figure}[H]
\includegraphics[width=0.75\textwidth]{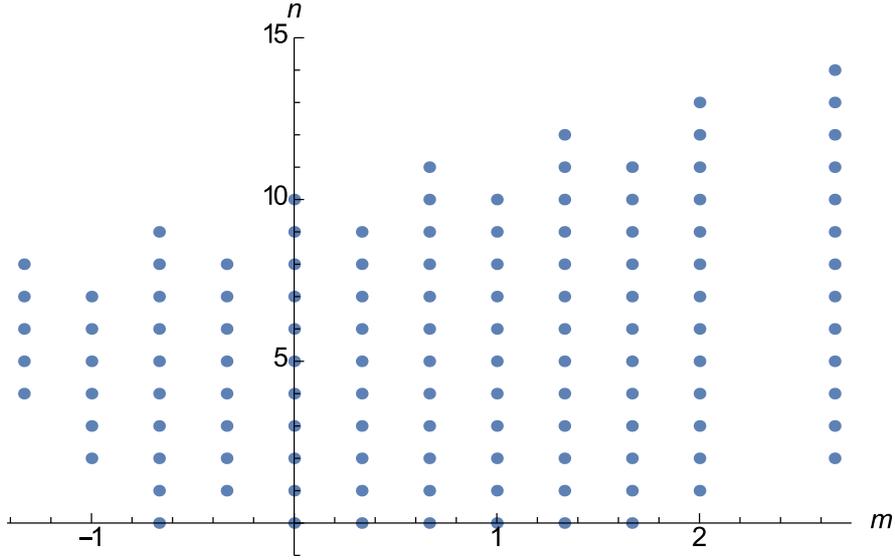}
\caption{Newton polygon of the exponents $ s^{m+n\delta_1} $ for the sub-leading term of asymptotic expansion of \eqref{eta0_alg+log}.}
\label{sub-lead_exponents-2}
\end{figure}

The set of admissible $ (m,n) $ components of the $s$-exponent is given in Fig. \ref{sub-lead_exponents-2} and the upper part of the 
convex hull of these points is defined by the seven points
$ 14 \delta _1+\frac{8}{3}, 13 \delta _1+2, 12 \delta _1+\frac{4}{3}, 11 \delta _1+\frac{2}{3}, 10 \delta _1, 9 \delta _1-\frac{2}{3}, 8 \delta _1-\frac{4}{3} $
and their mutual solution yields $ \delta_1=-2/3 $.

\begin{figure}[H]
\includegraphics[width=0.5\textwidth]{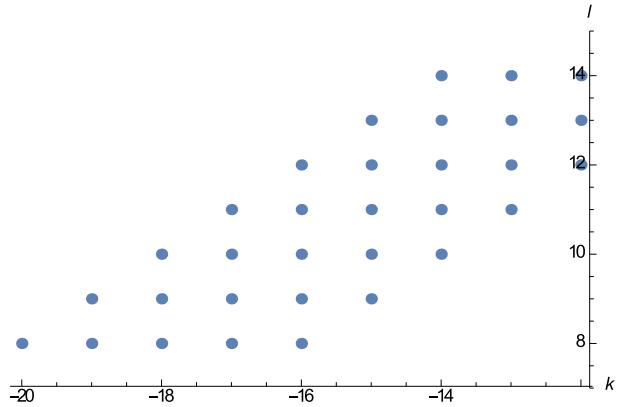}
\caption{Newton polygon of the exponents $ t^{k+l\mu_1} $ for the sub-leading term of asymptotic expansion of \eqref{eta0_alg+log}.}
\label{sub-lead_exponents-3}
\end{figure}

Given the above solution for $ \delta_1 $ we consider next the admissible $ (k,l) $ components of the $t$-exponent which are given in
Fig. \ref{sub-lead_exponents-3}. There are two lines defining the upper part of the convex hull of these points, however only the one defined
by the seven points
$ 14 \mu _1-14, 13 \mu _1-15, 12 \mu _1-16, 11 \mu _1-17, 10 \mu _1-18, 9 \mu _1-19, 8 \mu _1-20 $
ensures a finite solution, namely $ \mu_1=-1 $. In this case we have non-zero solutions for $ C_2 $.
\end{proof}

\begin{corollary}\label{asymptoticGap}
For large $s$ the $M=2$ gap probability $E_2 := E_{\nu_1,\nu_2}$ has the asymptotic form
\begin{equation}
E_{\nu_1,\nu_2}(0;(0,s)) = e^{-\frac{9}{2^{7/3}}s^{2/3} + O(s^{1/3})} .
\label{asymptoticE}
\end{equation}
\end{corollary}

\begin{remark}
It has been shown in \cite{FLZ_2015}, Eq. A.2, that the jpdf of Theorem \ref{GinibreProdJPDF} in the 
large separation limit takes on the simpler functional form proportional to
\begin{equation}
  \prod_{l=1}^N x_l^{-1/2+1/2M} e^{-M x_l^{1/M}} \prod_{1 \le j < k \le N}(x_k - x_j) (x_k^{1/M} - x_j^{1/M}) .
\end{equation}
After the change of variables $x_l \mapsto M x_l^M$ this specifies the Laguerre Muttalib-Borodin model: see Section 3.4 below. 
A key feature for present purposes is that exponentiating the product of differences gives the logarithmic pair potential
$V_2(x,y) = - \log(|x-y| |x^{1/2} - y^{1/2}|)$, which is scale invariant under multiplication
of the coordinates. According to Eq.~(14.117) of \cite{For_2010} this, together with the fact from Eq.~(5.15)
of \cite{For_2014a} that the hard edge spectral density is proportional to $1/x^{M/(M+1)}$, tells us that the leading $s \to \infty$
form of the gap probability at the hard edge is given by $e^{-C_M x^{2/(M+1)}}$
for some $C_M$. Our analytic result for $M=2$ (\ref{asymptoticE}) agrees with this predicted form.
\end{remark}

\subsection{\texorpdfstring{$ \theta=2 $}{MBE} Muttalib-Borodin Ensembles}

The Laguerre Muttalib-Borodin model refers to the eigenvalue PDF proportional to
\begin{equation*}
   \prod^{N}_{l=1}x_{l}^{c}e^{-x_l} \prod_{1\leq j<k\leq N}(x_j-x_k)(x_j^{\theta}-x_k^{\theta}), \qquad x_l \in \R_{>0} .
\end{equation*}
This is a determinantal point process, and so is fully specified by a correlation kernel, $ K^{L}(x,y) $ say. 
Define the hard edge scaling limit by
\begin{equation*}
  K^{(c,\theta)}(x,y) \coloneqq \lim_{N \to \infty}N^{-1/\theta}K^{L}(N^{-1/\theta}x,N^{-1/\theta}y) .
\end{equation*}
Borodin \cite{Bor_1999a} has obtained the evaluation
\begin{equation*}
  K^{(c,\theta)}(x,y) = \theta x^c \int^1_0 J_{\frac{c+1}{\theta},\frac{1}{\theta}}(xu)J_{c+1,\theta}((yu)^{\theta})u^c\;du ,
\end{equation*}
where the function $ J_{a,b}(x) $ defines the Wright Bessel function
\begin{equation*}
  J_{a,b}(x) \coloneqq \sum^{\infty}_{j=0} \frac{(-x)^j}{j!\Gamma(a+jb)} .
\end{equation*}
In a shift of notation we write $ K_M(x,y) $ defined by \eqref{HEkernel} as $ K_{\nu_1,\ldots,\nu_M}(x,y) $ 
to emphasize the dependency on the parameter set, and similarly
write $ E_{\nu_1,\dots,\nu_M}(0;(0,s)) $ for the scaled gap probability. These are well defined for all $ \nu_i > -1 $. 
We know from Kuijlaars and Zhang \cite{KZ_2014} and from Forrester and Wang \cite{FWang_2015} 
(see Eqs. (1.1), (1.5) and (5.8)) that for $ \theta \in \Z_{+} $
\begin{equation*}
  x^{1/\theta-1}K^{(c,\theta)}(\theta x^{1/\theta},\theta y^{1/\theta}) = K_{\nu_1,\ldots,\nu_{\theta}}(x,y) ,
\end{equation*}
where
\begin{equation*}
  \nu_j = \frac{c+j}{\theta}-1 , \qquad 1 \leq j \leq \theta.
\end{equation*}
Thus we can deduce that the gap probabilities \eqref{tauM} for $ c=0,1 $ and $ \theta=2 $ satisfies the identities
\begin{equation}
   E_{-1/2,0}(0;(0,s)) = E^{(0,2)}(0;(0,2\sqrt{s})) ,
\qquad
   E_{0,1/2}(0;(0,s))  = E^{(1,2)}(0;(0,2\sqrt{s})) ,
\label{GapProbId}
\end{equation}
where $E^{(c,\theta)}$ denotes the gap probability for the hard edge scaled Laguerre Muttalib-Borodin model.

The significance of this is that the kernels $ K^{(c,2)} $ are analytic, so we can apply Bornemann's numerical scheme 
\cite{Bor_2010}, \cite{Bor_2011a} to evaluate the gap probabilities in the large $ s $ regime and test numerically the 
asymptotic behaviour given in Prop. \ref{sLarge}. In this situation Bornemann's method converges exponentially fast 
and we can obtain accurate values for the gap probabilities in this regime. We have implemented the Bornemann method 
employing 9 nodes in the Clenshaw-Curtis quadrature rule with a precision of 20 decimal digits and truncating the 
Wright Bessel function series at 100 terms. A table of $ \log E^{(c,2)} $ versus $ r=2\sqrt{s} $ is given in the first columns of 
Table \ref{logE_and_a1Values} for both cases. We then compute a fit of $ \log E^{(c,2)} $ given on a range of $ r $ 
values to the assumed form $ a_1 r^{4/3} + b_1 r^{2/3} + c_1 $. The values of $ a_1 $ are tabulated for the range of 
$ r=4,\ldots,14 $ and the extrapolated value in the second columns of Table \ref{logE_and_a1Values} for both $ c=0,1 $.
This range was chosen, that is to say limited to these values, because at $ r=15 $ the value of $ E^{(0,2)} $ is already 
$ \text{8.917166}\times 10^{-15} $ and larger values are unreliable due to underflow. The extrapolated values should be 
compared to the predicted value (see \eqref{eta0_Large_s}) of $ 9\cdot 2^{-11/3} \sim 0.708705590566 $.

\begin{table}[H]
\renewcommand{\arraystretch}{1.2}
\begin{tabular}{|c|c|c|c|c|}
\hline
 & \multicolumn{2}{|c|}{$ c=0 $} & \multicolumn{2}{|c|}{$ c=1 $}  \\
\hline
 $ r $  & $ \log E^{(0,2)}(0;(0,r)) $ & $ a_1 $ & $ \log E^{(1,2)}(0;(0,r)) $ & $ a_1 $ \\
\hline
	 4 	&  -5.96549338586 	& -0.70729888196 	& -3.2910182568186667	& -0.7050253349947	\cr
	 5 	& -7.7702165574578 	& -0.707506362179	& -4.6175115857278 		& -0.7059621770238 	\cr
	 6 	& -9.666703768133 	& -0.707671636400	& -6.06617567204249		& -0.7065523127608	\cr
	 7 	& -11.6460744648319 & -0.707802917979	& -7.6216467824166		& -0.706953478338 	\cr
	 8 	& -13.701343595761 	& -0.707908414200 	& -9.2725398209570		& -0.707241379027 	\cr
	 9 	& -15.826846765594 	& -0.70799443184 	& -11.010033902389		& -0.70745656369 	\cr
	10 	& -18.017880484821 	& -0.70806558203	& -12.8270650595890		& -0.7076225663		\cr
	11 	& -20.27046470121 	& -0.7081252605		& -14.7178300927		& -0.7077538862		\cr
	12 	& -22.58117923782 	& -0.7081762248		& -16.6774638565		& -0.7078597927		\cr
	13 	& -24.9470471656 	& -0.7082218856		& -18.70181973197		& -0.7079460684		\cr
	14 	& -27.3654492473 	& -0.70827084		& -20.7873147490		& -0.7080153465 	\cr
$\infty$ & 					& -0.7088 			& 						& -0.7083172		\cr 
\hline
\end{tabular}
\vskip5mm
\caption{Computed values of $ \log E^{(c,2)}(0;(0,r)) $ and the coefficient $ a_1 $	versus $ r $ for $ c=0,1 $ 
and extrapolated values of $ a_1 $.}
\label{logE_and_a1Values}
\end{table}

We can also independently check the small $ s $ expansions generated by the non-linear
analogue of the $ \sigma$-form \eqref{eta0_ODE} by arbitrary high-accuracy small $ s $ expansions using the
Neumann expansions of the right-hand sides of \eqref{GapProbId}.
For $ \nu_1 =-1/2, \nu_2=0 $ we compute that the initial terms are
\begin{multline*}
  \eta_0(s) = -2\frac{\sqrt{s}}{\sqrt{\pi}}-2(4-\pi)\frac{s}{\pi}
              -\frac{32}{3}(3-\pi)\frac{s^{3/2}}{\pi^{3/2}}-\frac{16 }{9} (72-32\pi+3\pi^2)\frac{s^2}{\pi^2}
\\
              -\frac{64}{45} (360-200\pi+27\pi^2) \frac{s^{5/2}}{\pi^{5/2}} 
              -\frac{512}{675} (2700-1800\pi+347\pi^2-15\pi^3) \frac{s^3}{\pi^3}  + {\rm O}(s^{7/2}) .
\end{multline*}
Using computer algebra, we have extended this series to high order, and have computed as well the series 
expansion of $F$ as implied by \eqref{Fsquared} to high order. We find the remarkable but not understood 
relation
\begin{equation}
   6 - 2F = \eta_0' .
\end{equation}
Substituting this in \eqref{Fsquared} we find the even more remarkable, and similarly not understood result that 
the resolvent function satisfies the much simpler third-order non-linear ODE
\begin{equation}
  -12 s^2\eta_0'\eta_0{}^{(3)}+9 s^2\eta_0''{}^2-12 s\eta_0'\eta_0''+\frac{3}{4}\eta_0' \left[\eta_0'(-48 s\eta_0'+16\eta_0+1)+4 \right]-9 = 0 .
\label{3rdOrderODE}
\end{equation}

\begin{remark}
We can check that \eqref{3rdOrderODE} is consistent with \eqref{eta0_Large_s}. It is furthermore the case that a
large $ s $ analysis of \eqref{3rdOrderODE} allows \eqref{asymptoticE}, with $ \nu_1 = -1/2$, $\nu_2 = 0 $ to be 
strengthened to read
\begin{equation}
  E_{-1/2,0}(0;(0,s)) = e^{-\frac{9}{2^{7/3}} s^{2/3} - \frac{3}{2^{5/3}} s^{1/3} + O(s^{1/6})} .
\label{}  
\end{equation}
\end{remark}

\subsection{Higher order analogues of Painlev\'e III}\label{4accessoryPainleve}

A program to enumerate all of the higher order analogues of the Painlev\'e equations, at least to the next level
of four-dimensional or four accessory parameters, has been initiated by H. Kawakami, A. Nakamura and H. Sakai 
in the period 2012-15. The first phase of the task was achieved with the construction of isomonodromy deformation problems for 
Fuchsian differential equations, extending the four singularity case corresponding to Painlev\'e VI, in \cite{Sak_2010} 
by the techniques of addition and middle convolution. This yielded the four master cases: the Garnier systems, the 
Fuji-Suzuki systems, the Sasano systems and the matrix Painlev\'e systems, which we tabulate below -

\bigskip
\begin{center}
\begin{minipage}{\textwidth}
\begin{xy}
{(0,0) *{\begin{tabular}{|c|}
\hline
1+1+1+1+1\\
\hline\hline
$11,11,11,11,11$\\
$H_{\mathrm{Garnier}}^{1+1+1+1+1}$\\
\hline
\end{tabular}
}}\end{xy}

\bigskip
\begin{xy}
{(0,0) *{\begin{tabular}{|c|}
\hline
1+1+1+1\\
\hline\hline
$21,21,111,111$\\
$H_{\mathrm{Fuji-Suzuki}}^{A_5}$\\
\hline
\end{tabular}
}}\end{xy}

\bigskip
\begin{xy}
{(0,0) *{\begin{tabular}{|c|}
\hline
1+1+1+1\\
\hline\hline
$31,22,22,1111$\\
$H_{\mathrm{Sasano}}^{D_6}$\\
\hline
\end{tabular}
}}
\end{xy}

\bigskip
\begin{xy}
{(0,0) *{\begin{tabular}{|c|}
\hline
1+1+1+1\\
\hline\hline
$22,22,22,211$\\
$H_{\rm VI}^{\mathrm{Matrix}}$\\
\hline
\end{tabular}
}}
\end{xy}

\end{minipage}
\end{center}
\bigskip

These four master cases were extended by constructing from them the degeneration schemes of singularity
confluence in \cite{KNS_2012} and \cite{KNS_2013}, and yields four families. Of the four families found 
the only family relevant to our case, certain higher order analogues of Painlev\'e III, is the Fuji-Suzuki family which
have $ 3\times 3 $ Lax pairs. There are nine cases in this degeneration scheme.

\bigskip
\begin{xy}
{(0,0) *{\begin{tabular}{|c|}
\hline
1+1+1+1\\
\hline\hline
$21,21,111,111$\\
$H_{\mathrm{FS}}^{A_5}$\\
\hline
\end{tabular}
}},
{\ar (12,-1);(22,10)},
{\ar (12,-1);(22,-1)},
{\ar (12,-1);(22,-12)},
{(35,0) *{\begin{tabular}{|c|}
\hline
2+1+1\\
\hline\hline
$(2)(1),111,111$\\
$H_{\mathrm{NY}}^{A_5}$\\
\hline
$(11)(1),21,111$\\
$H_{\mathrm{FS}}^{A_4}$\\
\hline
$(1)(1)(1),21,21$\\
$H_{\mathrm{Gar}}^{2+1+1+1}$\\
\hline
\end{tabular}
}},
{\ar (48,10);(56,22)},
{\ar (48,10);(56,-23)},
{\ar (48,-1);(56,22)},
{\ar (48,-1);(56,10)},
{\ar (48,-1);(56,-12)},
{\ar (48,-11);(56,10)},
{\ar (48,-11);(56,-23)},
{(70,18) *{\begin{tabular}{|c|}
\hline
3+1\\
\hline\hline
$((11))((1)),111$\\
$H_{\mathrm{NY}}^{A_4}$\\
\hline
$((1)(1))((1)),21$\\
$H_{\mathrm{Gar}}^{3+1+1}$\\
\hline
\end{tabular}}},
{(70,-18) *{\begin{tabular}{|c|}
\hline
2+2\\
\hline\hline
$(11)(1),(11)(1)$\\
$H_{\mathrm{FS}}^{A_3}$\\
\hline
$(2)(1),(1)(1)(1)$\\
$H_{\mathrm{Gar}}^{\frac{3}{2}+1+1+1}$\\
\hline
\end{tabular}}},
{(105,0) *{\begin{tabular}{|c|}
\hline
4\\
\hline\hline
$(((1)(1)))(((1)))$\\
$H_{\mathrm{Gar}}^{\frac{5}{2}+1+1}$\\
\hline
\end{tabular}}},
{\ar (84,22);(91,-1)},
{\ar (84,10);(91,-1)},
{\ar (84,-12);(91,-1)},
{\ar (84,-25);(91,-1)},
\end{xy}
\bigskip

However such a classification treats only the unramified cases and only very recently have ramified cases
been studied, and a partial list of results has been given in \cite{Kaw_2015}. In addition to the nine shown
above another seven ramified cases are given. However of those only one is a possibility, namely the one
with the singularity pattern $ \frac{4}{3}+1+1 $ and spectral type $ (1)_{3},21,111 $ and has a Riemann-Papperitz
symbol
\begin{equation}
   \left\{ \begin{array}{cccc}
            0 		& 1 		& \multicolumn{2}{c}{\infty(\frac{1}{3})} \\
            0 		& 0 		& t^{1/3} 		& \theta^{\infty}_1/3-\frac{2}{3} \\
            \theta^0_1 	& 0 		& \omega t^{1/3} 	& \theta^{\infty}_1/3-\frac{2}{3}  \\
            \theta^0_2 	& \theta^1 	& \omega^2 t^{1/3} 	& \theta^{\infty}_1/3-\frac{2}{3} 
           \end{array}
   \right\} ,
\label{F-S_4.3}
\end{equation} 
with $ \theta^0_1+\theta^0_2+\theta^1+\theta^{\infty}_1 = 0 $ and $ \omega^3 = 1 $. The comparison that must be 
made here is with our system \eqref{R-PsymbolM=2}, and there are several differences to note. One is that while
the indicial exponents at the $ z=1 $ singularity of \eqref{R-PsymbolM=2} are all zero (only two are independent)
this is just an artefact of the Fredholm theory, which always leads to these exponents vanishing whereas the 
general integrable system possesses a full set of exponents. Thus we suspect that the generalisation of our
system actually has one or possibly two additional, free non-zero parameters here and thus either two or all three 
are different. However in 
\eqref{F-S_4.3} two of the parameters are locked together (here they are conventionally set to zero). 
Another difference arises also, where the sub-leading spectral data at $ z=\infty $ are all equal, whereas in our
application these are not equal even in special cases. In summary we believe that there are additional ramified
cases to be found in the Fuji-Suzuki family, and that our system is a special case of one such system. Such a
system might arise from the unramified system with singularity pattern $ 2+1+1 $ and spectral type $ (2)(1),111,111 $
by a transition involving a fractional drop in the Poincar\'e index.

\section*{Acknowledgements}
The work of NSW and PJF was supported by the Australian Research Council Discovery Project DP140102613.

\section*{Appendix}
\renewcommand{\theequation}{A-\arabic{equation}} 		
\renewcommand{\theproposition}{A-\arabic{proposition}} 	
\setcounter{equation}{0}  								
\setcounter{proposition}{0}  							

Here the Hamiltonian variables in the case $M=2$ are expressed in terms of $\eta_0$ and its derivatives.

\begin{proposition}
All the dynamical variables can be recovered from the resolvent function $ \eta_0 $ and its derivatives in the following list
of formulae. Here $ F $ should be interpreted as the positive square root of \eqref{Fsquared}.
\begin{multline}
  \xi_0 = 
  - (3+e_1-3 \eta_0) \frac{\left(e_1^2 (-F)+3 e_2 F+3 (9-F) \eta_0\right)}{162 \eta_0'}
\\
  + \frac{1}{162} \left[ 9 e_1 \left(3 e_2 \left(\eta_0-1\right)+3 e_3+(3-F) s-3 \left(\eta_0-1\right) \eta_0\right) \right. 
\\ \left.
   +27 \left(\eta_0 \left(4 e_2-(3-F) s+\left(\eta_0-2\right) \eta_0+1\right)-3 e_3 \left(\eta_0+1\right)+3 s\right)
   -6 e_1^3 \left(\eta_0-1\right)-9 e_1^2 \left(e_2+2 \eta_0\right)+2 e_1^4 \right]
\\
  - \frac{1}{6} s \left(e_1-3 \eta_0+1\right) \eta_0'
\\
  + \left[ (3+e_1-3 \eta_0) \frac{s}{108 \eta_0'{}^2} \left(\frac{36 s \eta_0'{}^3}{F}+F\right)+\frac{s^2}{6} \right]\eta_0'' 
\\
  + (3+e_1-3 \eta_0) \left[ -\frac{s^2 \left(e_1^2-3 e_2+3 \eta_0-3\right)}{18 F \eta_0'}+\frac{s^3}{F}-\frac{F s^2}{72 \eta_0'{}^3} \right]\eta_0''{}^2 
\\
  - (3+e_1-3 \eta_0) \frac{s^3\eta_0''{}^3}{4 F \eta_0'{}^2}
  + (3+e_1-3 \eta_0) \frac{s^4\eta_0''{}^4}{8 F \eta_0'{}^3}
\\
  + (3+e_1-3 \eta_0) \left[ -\frac{s^4 \eta_0''{}^2}{4 F \eta_0'{}^2}+\frac{s^3 \eta_0''}{2 F \eta_0'}+\frac{F s^2}{108 \eta_0'{}^2} \right]\eta_0{}^{(3)}
  + (3+e_1-3 \eta_0) \frac{s^4 \eta_0''\eta_0{}^{(4)}}{6 F \eta_0'} ,
\label{Rep:xi0}
\end{multline}
\begin{multline}
  \xi_1 = 
  +\frac{e_1^2 (-F)+3 e_2 F+3 (9-F) \eta_0}{54 \eta_0'}
\\
  +\frac{1}{54} \left[ -9 \left(-6 e_2+3 e_3+(3-F) s-3 \left(\eta_0-1\right) \eta_0\right)+9 e_1 \left(e_2-4 \eta_0\right)-2 e_1^3 \right]
\\
  +\frac{1}{2} s \eta_0'
  -\left[ \frac{s^2 \eta_0'}{F}+\frac{F s}{36 \eta_0'{}^2} \right]\eta_0'' 
\\ 
  +\left[ \frac{s^2 \left(e_1^2-3 e_2+3 \eta_0-3\right)}{6 F \eta_0'}-\frac{3 s^3}{F}+\frac{F s^2}{24 \eta_0'{}^3} \right]\eta_0''{}^2 
\\
  +\frac{3 s^3 \eta_0''{}^3}{4 F \eta_0'{}^2}
  -\frac{3 s^4 \eta_0''{}^4}{8 F \eta_0'{}^3}
  -\left[ -\frac{3 s^4 \eta_0''{}^2}{4 F \eta_0'{}^2}+\frac{3 s^3 \eta_0''}{2 F \eta_0'}+\frac{F s^2}{36 \eta_0'{}^2} \right]\eta_0{}^{(3)} 
  -\frac{s^4\eta_0''}{2 F \eta_0'} \eta_0{}^{(4)} , 
\label{Rep:xi1}
\end{multline}
\begin{multline}
  \eta_1 = 
+\frac{(9-F) \eta_0}{18 \eta_0'}
\\
+\frac{1}{54} \left[ -9 \left(3 e_3+(3-F) s+3 \left(\eta_0-1\right) \eta_0\right)+9 e_1 \left(e_2+2 \eta_0\right)-2 e_1^3 \right]
\\
-\frac{1}{2} s\eta_0'
-\left(e_1^2-3 e_2\right) \left(e_1^2-3 e_2+3 \eta_0\right) \frac{2}{27 F}\eta_0'
+ \left(e_1^2-3 e_2\right) \frac{2s}{3 F} \eta_0'{}^2
\\
+ \left(e_1^2-3 \left(e_2+\eta_0\right)\right) \frac{s\eta_0''}{9 F}
\\ 
\left[ \frac{s^2 \left(e_1^2-3 e_2+6 \eta_0-1\right)}{6 F \eta_0'}-\frac{9 s^3}{2 F} \right]\eta_0''{}^2 
\\
+\left[ \frac{s^2 \left(e_1^2-3 \left(e_2+\eta_0\right)\right)}{9 F}-\frac{5 s^3 \eta_0''}{6 F \eta_0'}+\frac{s^3 \eta_0'}{F} \right]\eta_0{}^{(3)} 
+\frac{s^4 \eta_0{}^{(3)}{}^2}{3 F \eta_0'}
-\frac{s^4 \eta_0{}^{(4)} \eta_0''}{2 F \eta_0'} ,
\label{Rep:eta1}
\end{multline}
\begin{multline}
  \eta_2 = 
    \frac{\eta_0 \left[ 9 \left(2 e_1-3 \eta_0+3\right)-F \left(2 e_1-3 \eta_0\right) \right]}{54 \eta_0'}
\\
  + \frac{1}{162} \Big[ -4 e_1^4+6 \left(\eta_0-1\right) e_1^3+18 \left(e_2+2 \eta_0\right) e_1^2-9 \left(2 (3-F) s+6 e_3+3 e_2 \left(\eta_0-1\right)+6 \left(\eta_0-1\right) \eta_0\right) e_1
\\
    +27 \left(-3 s+3 e_3 \left(\eta_0-1\right)+\eta_0 \left(3 s-2 e_2+\left(\eta_0-2\right) \eta_0+1\right)\right) \Big]
\\
  - \frac{1}{162 F} \Big[ 8 e_1^5-12 \left(\eta_0-1\right) e_1^4+24 \left(\eta_0-2 e_2\right) e_1^3+36 \left(2 e_2 \left(\eta_0-1\right)-\left(\eta_0-2\right) \eta_0\right) e_1^2 
\\
    -18 \left(4 e_2 \left(\eta_0-e_2\right)-3 F s\right) e_1+27 \left(-4 \left(e_2-\eta_0\right) \left(e_2 \left(\eta_0-1\right)+\eta_0\right)-F s \left(3 \eta_0-1\right)\right) \Big] \eta_0' 
\\
  - \frac{2 \left(9 \eta_0{}^2+3 \left(2 e_1^2-6 e_2-3\right) \eta_0-\left(2 e_1+3\right) \left(e_1^2-3 e_2\right)\right) \eta_0'{}^2 s}{9 F}
  + \frac{6 \eta_0 \eta_0'{}^3 s^2}{F}
\\
  + \left[ \frac{2 s^2 \eta_0 \eta_0'}{F}-\frac{s \left(-9 F s-12 \left(2 e_1+3\right)+\left(2 e_1-3 \eta_0+3\right) \left(6 \left(e_2+\eta_0+2\right)-2 e_1^2\right)\right)}{54 F} \right] \eta_0''
\\
  + \left[ -\frac{3 \left(2 e_1-2 \eta_0+3\right) s^3}{2 F}-\frac{\left(6 e_1+\left(2 e_1-3 \eta_0+3\right) \left(-e_1^2+3 e_2-6 \eta_0-2\right)+9\right) s^2}{18 F \eta_0'} \right] \eta_0''{}^2
\\
  + \left[ \frac{\left(2 e_1+3 \eta_0+3\right) \eta_0' s^3}{3 F}-\frac{5 \left(2 e_1-3 \eta_0+3\right) \eta_0'' s^3}{18 F \eta_0'} \right. 
\\ \left.
         -\frac{\left(\left(3 \left(e_2+\eta_0+2\right)-e_1^2\right) \left(2 e_1-3 \eta_0+3\right)-6 \left(2 e_1+3\right)\right) s^2}{27 F} \right] \eta_0{}^{(3)}
\\
  + \left(2 e_1-3 \eta_0+3\right) \frac{\eta_0{}^{(3)}{}^2 s^4}{9 F \eta_0'}-\left(2 e_1-3 \eta_0+3\right) \frac{\eta_0'' \eta_0{}^{(4)} s^4}{6 F \eta_0'} ,
\label{Rep:xi2}
\end{multline}
\begin{equation}
 x_0y_1 = \frac{1}{6} \left[ -F+4 e_1 \eta_0'-6 \eta_0 \eta_0'-3 s \eta_0'' \right)] ,
\label{Rep:x0y1}
\end{equation}
\begin{equation}
 x_1y_2 = \frac{1}{6} \left[ -F-2 e_1 \eta_0'+6 \eta_0 \eta_0'+3 s \eta_0'' \right] ,
\label{Rep:x1y2}
\end{equation}
\begin{equation}
 x_0y_2 = -\eta_0' ,
\label{Rep:x0y2}
\end{equation} 
\begin{multline}
 x_0y_0 =  \frac{1}{54} \left[ \left(e_1 \left(e_1+3\right)-3 e_2\right) F-3 (2 F+9) \eta_0 \right]
\\
 +\frac{1}{54} \eta_0' \left[ 9 \left(2 e_1+1\right) \eta_0+9 \left(3 \left(e_3+s\right)-4 e_2\right)+e_1 \left(2 e_1 \left(e_1+3\right)-9 e_2\right)-9 F s-27 \eta_0{}^2 \right]
\\
  -\frac{1}{2} s \eta_0'{}^2+\frac{2 s \eta_0'{}^3}{F}
\\
 +\eta_0'' \left[ -\frac{s \left(e_1^2-3 e_2+3 \eta_0-3\right) \eta_0'}{3 F}+\frac{1}{6} s \left(e_1-3 \eta_0-1\right)+\frac{7 s^2 \eta_0'{}^2}{F}-\frac{F s}{18 \eta_0'} \right]
\\
 +\eta_0''{}^2 \left[ -\frac{s^2 \left(e_1^2-3 e_2+3 \eta_0+6\right)}{6 F}+\frac{3 s^3 \eta_0'}{F}-\frac{F s^2}{24 \eta_0'{}^2} \right]
 +\frac{3 s^4 \eta_0''{}^4}{8 F \eta_0'{}^2}
\\
 +\eta_0{}^{(3)} \left[ -\frac{3 s^4 \eta_0''{}^2}{4 F \eta_0'}+\frac{3 s^2 \eta_0'}{F}+\frac{F s^2}{36 \eta_0'}-\frac{s^2}{6} \right]
 +\eta_0{}^{(4)} \left[ \frac{s^4 \eta_0''}{2 F}+\frac{s^3 \eta_0'}{F} \right] ,
\label{Rep:x0y0}
\end{multline}
\begin{multline}
 x_1y_1 = \frac{1}{18}e_1 F+\frac{1}{9} \left[ -3 \left(3 e_1+1\right) \eta_0+e_1^2+3 e_2+9 \eta_0{}^2 \right] \eta_0'
\\
         +s \eta_0'{}^2+\frac{1}{6} s (2-3 e_1+6 \eta_0) \eta_0''+\frac{1}{3} s^2 \eta_0{}^{(3)} ,
\label{Rep:x1y1}
\end{multline} 
\begin{multline}
 x_2y_2 = 
  \frac{1}{54} \left[ -e_1 \left(e_1+6\right)F+3 e_2 F+3 (2 F+9) \eta_0 \right]
\\
 +\frac{1}{54} \eta_0' \left[ 9 \left(2 e_2-3 e_3+(F-3) s-3 \eta_0{}^2+\eta_0\right)+9 e_1 \left(e_2+4 \eta_0\right)-2 e_1^3-12 e_1^2 \right]
\\
 -\frac{1}{2} s \eta_0'{}^2-\frac{2 s \eta_0'{}^3}{F}
\\
 +\eta_0'' \left[ \frac{s \left(e_1^2-3 e_2+3 \eta_0-3\right) \eta_0'}{3 F}+\frac{1}{6} s \left(2 e_1-3 \eta_0-1\right)-\frac{7 s^2 \eta_0'{}^2}{F}+\frac{F s}{18 \eta_0'} \right]
\\
 +\eta_0''{}^2 \left[ \frac{s^2 \left(e_1^2-3 e_2+3 \eta_0+6\right)}{6 F}-\frac{3 s^3 \eta_0'}{F}+\frac{F s^2}{24 \eta_0'{}^2} \right]
 -\frac{3 s^4 \eta_0''{}^4}{8 F \eta_0'{}^2}
\\
 +\eta_0{}^{(3)} \left[ \frac{3 s^4 \eta_0''{}^2}{4 F \eta_0'}-\frac{3 s^2 \eta_0'}{F}-\frac{F s^2}{36 \eta_0'}-\frac{s^2}{6} \right]
 +\eta_0{}^{(4)} \left[ -\frac{s^4 \eta_0''}{2 F}-\frac{s^3 \eta_0'}{F} \right] .
\label{Rep:x2y2}
\end{multline}
\end{proposition}
\begin{proof}
We employ the abbreviations for $ U, V $ in \eqref{alt-6th} and \eqref{alt-7th}, and together with
\begin{equation}
    \xi_{1}-\eta_{1} = e_2+\eta_{0}(\eta_0-e_1)-\eta_{0}+s\eta_0^{\prime} ,
\label{xi1-eta1}
\end{equation} 
and \eqref{alt_Ham}, we have a system of four linear, independent equations for $ \xi_0, \xi_1, \eta_1, \eta_2 $
in terms of $ U, V, W, Z $, and $ \eta_0 $ and its derivatives. For the bilinear products we will use the formulae
\eqref{aux:x0y1}, \eqref{aux:x1y2}, \eqref{aux:x0y0} and \eqref{aux:x2y2}.
The next step is to solve for $ U, V, W, Z $ and in contrast to the proof of \eqref{eta0_ODE} we employ 
\eqref{UVsum}, \eqref{FUV}, \eqref{WZsum} and \eqref{WZcombo}. After some simplifying we arrive at 
\eqref{Rep:xi0}-\eqref{Rep:x2y2}.
\end{proof}

One final result should be stated here and this concerns the splitting of $ x_0y_2 $ and involves the introduction
of a decoupling factor $ G $ such that $ x_0 \coloneqq y_2G $. For $ M=1 $ this was a simple algebraic factor but
for $ M \geq 2 $ this is no longer the case.

\begin{proposition}
The decoupling factor $ G $ satisfies the first-order ordinary differential equation
\begin{equation}
 \left( 3s\frac{G'}{G}+2e_1 \right)^2 - 4e_1^2 = 12\left( \eta_0-e_2-3s\eta_0'-s\frac{\eta_0''}{\eta_0'} \right)
                                            - 12s^2\left( \frac{\eta_0{}^{(3)}}{\eta_0'}-\frac{3}{4}\left( \frac{\eta_0''}{\eta_0'} \right)^2 \right) ,
\label{GODE}
\end{equation} 
and the boundary condition as $ s\to 0 $
\begin{equation}
  G^{-1} \sim -\Gamma(\nu_2-\nu_1)\Gamma(\nu_2-\nu_0+1)s^{\nu_1+\nu_0}-\Gamma(\nu_1-\nu_2)\Gamma(\nu_1-\nu_0+1)s^{\nu_2+\nu_0} .
\end{equation} 
\end{proposition}
\begin{proof}
Clearly $ U=sy_2y_2'G $ and $ V=sy_2^2G'+sy_2y_2'G $, and together with \eqref{FUV} and \eqref{Fsquared}, we 
deduce \eqref{GODE}. The boundary condition is a consequence of \eqref{x0_IC} and \eqref{y2_IC}. 
\end{proof}

\bibliographystyle{plain}
\bibliography{moment,random_matrices,nonlinear,CA}

\end{document}